\documentclass[11pt]{article}
\usepackage{amsmath}
\usepackage{amssymb}
\usepackage{amsthm} 
\usepackage{bbm}
\usepackage{graphicx}
\usepackage{algorithm}
\usepackage{algorithmic}
\usepackage{fullpage}
\usepackage{stmaryrd}
\usepackage{color,url}
\usepackage{hyperref}
\usepackage{caption}
\usepackage{subcaption}
\usepackage{hyperref}
\usepackage{cleveref}
\usepackage{todonotes}
\usepackage{xpatch}
\usepackage{silence}

\usepackage{pgfplots}
\pgfplotsset{compat=1.14}
\pgfplotsset{plot coordinates/math parser=false}

\usepackage{placeins} %

\crefname{algorithm}{algorithm}{algorithms}
\Crefname{algorithm}{Algorithm}{Algorithms}
\makeatletter
\newcommand*{\algotitle}[2]{%
  \stepcounter{algocf}%
  \hypertarget{algocf.title.\theHalgocf}{}%
  \NR@gettitle{#1}%
  \label{#2}%
  \addtocounter{algocf}{-1}%
}
\def\imod#1{\allowbreak\mkern10mu({\operator@font mod}\,\,#1)}
\makeatother

\newtheorem{cor}{Corollary}
\newtheorem{thm}{Theorem}
\newtheorem{lem}{Lemma}

\newtheorem{rem}{Remark}

\def\x{{\mathbf x}}

\def\z{{\mathbf z}}

\def\f{{\mathbf f}}
\def\v{{\mathbf v}}
\def\a{{\mathbf a}}
\def\k{{\mathbf k}}
\def\h{{\mathbf h}}
\def\F{{\mathbf F}}

\DeclareMathOperator{\supp}{supp}

\makeatletter
\xpatchcmd{\algorithmic}%
	{\arabic{ALC@line}}%
	{(\Roman{ALC@line})}%
	{}{}
	\renewcommand{\theALC@line}{(\Roman{ALC@line})}
\makeatother

\begin{document}
\newenvironment{frcseries}{\fontfamily{frc} \selectfont}{}
\newcommand{\textfrc}[1]{{\frcseries #1}}
\newcommand{\mathfrc}[1]{\text{\textfrc{#1}}}
\newcommand{\C}{\ensuremath{\mathbbm{C}}}
\renewcommand{\algorithmicrequire}{\textbf{Input:}}
\renewcommand{\algorithmicensure}{\textbf{Output:}}
\algsetup{indent=0.25in,linenodelimiter=,linenosize=\small}

\title{Sparse Fourier Transforms on Rank-1 Lattices for the Rapid and Low-Memory Approximation of Functions of Many Variables}
\author{Craig Gross\thanks{Michigan State University, Department of Mathematics, \texttt{grosscra@msu.edu}.  Supported in part by NSF DMS 1912706.} 
\and 
Mark Iwen\thanks{Michigan State University, Department of Mathematics, and the Department of Computational Mathematics, Science and Engineering (CMSE), \texttt{markiwen@math.msu.edu}.  Supported in part by NSF DMS 1912706.}
\and
Lutz K\"{a}mmerer\thanks{Chemnitz University of Technology, Faculty of Mathematics, \texttt{lutz.kaemmerer@mathematik.tu-chemnitz.de}.}
\and 
Toni Volkmer\thanks{Chemnitz University of Technology, Faculty of Mathematics, \texttt{toni.volkmer@math.tu-chemnitz.de}.  Supported in part by S\"achsische Aufbaubank -- F\"orderbank -- (SAB) 100378180.} 
}

\maketitle

\begin{abstract}
This paper considers fast and provably accurate algorithms for approximating smooth functions on the $d$-dimensional torus, $f: \mathbbm{ T }^d \rightarrow \mathbbm{C}$, that are sparse (or compressible) in the multidimensional Fourier basis.  In particular, suppose that the Fourier series coefficients of $f$, $\{c_{\bf k} (f) \}_{{\bf k} \in \mathbbm{Z}^d}$, are concentrated in a given arbitrary finite set $\mathcal{I} \subset \mathbbm{Z}^d$ so that
$$\min_{\Omega \subset \mathcal{I} ~s.t.~ \left| \Omega \right| =s }\left \| f - \sum_{{\bf k} \in \Omega} c_{\bf k} (f) \, \mathbbm{ e }^{ -2 \pi \mathbbm{ i } \k \cdot \circ} \right\|_2 < \epsilon \|f \|_2$$ holds for $s \ll \left| \mathcal{I} \right|$ and $\epsilon \in (0,1)$ small.  In such cases we aim to both identify a near-minimizing subset $\Omega \subset \mathcal{I}$ and accurately approximate its associated Fourier coefficients $\{ c_{\bf k} (f) \}_{{\bf k} \in \Omega}$ as rapidly as possible.  In this paper we present both deterministic and explicit as well as randomized algorithms for solving this problem using $\mathcal{O}(s^2 d \log^c (|\mathcal{I}|))$-time/memory and $\mathcal{O}(s d \log^c (|\mathcal{I}|))$-time/memory, respectively.  Most crucially, all of the methods proposed herein achieve these runtimes while simultaneously satisfying theoretical best $s$-term approximation guarantees which guarantee their numerical accuracy and robustness to noise for general functions.

These results are achieved by modifying several different one-dimensional Sparse Fourier Transform (SFT) methods to subsample a function along a reconstructing rank-1 lattice for the given frequency set $\mathcal{I} \subset \mathbbm{Z}^d$ in order to rapidly identify a near-minimizing subset $\Omega \subset \mathcal{I}$ as above without having use anything about the lattice beyond its generating vector.  This requires the development of new fast and low-memory frequency identification techniques capable of rapidly recovering vector-valued frequencies in $\mathbbm{Z}^d$ as opposed to recovering simple integer frequencies as required in the univariate setting.  Two different multivariate frequency identification strategies are proposed, analyzed, and shown to lead to their own best $s$-term approximation methods herein, each with different accuracy versus computational speed and memory tradeoffs.


\end{abstract}

\section{Introduction}

This paper considers methods for efficiently computing sparse Fourier transforms of multivariate periodic functions using rank-1 lattices.
In particular, for a function $ f: \mathbbm{ T }^d \rightarrow \C $ (where $ \mathbbm{ T } := [0, 1] $ with the endpoints identified), our goal is to compute the Fourier coefficients of $ f $,
\begin{equation*}
	c_\k (f) := \int_{ \mathbbm{ T }^d } f(\x) \, \mathbbm{ e }^{ -2 \pi \mathbbm{ i } \k \cdot \x } \; \mathrm{d} \x,
\end{equation*}
via samples of $ f $ at points in $ \mathbbm{ T }^d $.
Here, we assume that $f$ is from the
Wiener algebra $ \mathcal{W}(\mathbbm{ T }^d) := \{ f \in L^1(\mathbbm{ T }^d) \colon \| f ||_{ \mathcal{W}(\mathbbm{ T }^d) } := \sum_{ \k \in \mathbbm{ Z }^d } |c_\k (f)| < \infty \} $,
and that $f$ is well approximated by just a few of the dominant terms in its Fourier expansion (i.e., has an accurate sparse approximation in the Fourier basis).

One quasi-Monte Carlo approach which is especially popular in the context of Fourier approximations is sampling along rank-1 lattices adapted to frequency spaces of interest \cite{Temlyakov1986,Temlyakov1993,Li2003,Kuo2006,MuntheKaas2012,kaemmererdiss,Kaemmerer2015,Plonka2018,Kuo2020}.
In the standard rank-1 lattice approach, a one-dimensional, length-$M$ discrete Fourier transform (DFT) is applied to samples of $ f $ along a rank-1 lattice $ \Lambda(\z, M) $ with generating vector $ \z \in \mathbbm{ Z }^d $ over $ \mathbbm{ T }^d $ defined by 
\begin{equation*}
	\Lambda(\z, M) := \left\{ \frac{j}{M}\z \bmod 1 \mid j \in [M] := \{0, 1, \ldots, M - 1\}  \right\}.
\end{equation*}
Writing any function $ f \in \mathcal{W}(\mathbbm{ T }^d) $ in terms of its Fourier series, $ f = \sum_{ \k \in \mathbbm{ Z }^d } c_\k(f) \, \mathbbm{ e }^{ 2 \pi \mathbbm{ i } \k \cdot \circ } $, the DFT of these samples along the lattice is exactly equivalent to the DFT of the univariate function
\begin{equation}
	\label{eq:LatticeFunction}
	a(t) := \sum_{ \k \in \mathbbm{ Z }^d } c_\k(f) \, \mathbbm{ e }^{2 \pi \mathbbm{ i } \k \cdot \z \, t }
\end{equation}
using the equispaced samples $ \big(a( j/M ) \big)_{ j \in [M] } = \big( f(\x) \big)_{ \x \in \Lambda(\z, M) } $.
Just as the DFT of equispaced samples of $ a $ can be used to approximate its Fourier coefficients, so then can this DFT be used to help approximate the original Fourier coefficients $c_\k(f)$ of $ f $.
Though the process of matching the discrete coefficients to corresponding coefficients of $ f $ is nontrivial (see the following paragraph for further discussion), this multivariate to univariate transformation allows us to carry over many standard one-dimensional DFT results in a straightforward manner.
In particular, under our assumption of $ f $ being Fourier sparse or compressible, one-dimensional sparse Fourier transform (SFT) techniques \cite{Gilbert2005,Gilbert2008,Iwen2010,Hassanieh2012,Christlieb2013,Segal2013,Iwen13,Gilbert2014,Christlieb2016,Kapralov2016,Plonka2017,Plonka2018Deterministic,Mehri2019,Bittens2019Sparse,Bittens2019Real} become particularly appealing as they can sidestep runtimes which depend polynomially on the bandwidth, in this case~$ M $, instead running sublinearly in the magnitude of the underlying frequency space under consideration.
Additionally, these techniques often furnish recovery guarantees for Fourier compressible functions in terms of best $ s $-term approximations in the same vein as compressed sensing results \cite{Cohen2009,Foucart2013}.

However, in order for our univariate DFT to be properly related to the original multivariate Fourier coefficients~$c_\k (f)$, any aliasing must not produce extraneous collisions which perturb the multivariate to univariate transformation.
Specifically, after applying a length-$ M $ DFT to the univariate function~$ a $ in \eqref{eq:LatticeFunction}, all one-dimensional frequencies $ \k \cdot \z $ are aliased to their residues modulo $ M $.
Restricting our attention to some finite multivariate frequency set $ \mathcal{I} \subset \mathbbm{ Z }^d $, we then consider \emph{reconstructing} rank-1 lattices for $ \mathcal{I} $, that is, lattices where the mapping $ m_{ \z, M }: \mathcal{I} \rightarrow [M]$ given by $ \k \mapsto \k \cdot \z \bmod M $ is injective.
In this case, each coefficient produced by the DFT of $ a $ can be uniquely mapped back to the corresponding multivariate frequency~$\k$ of $ f $ by inverting $ m_{ \z, M } $.

In order to know or store this inverse map we require the calculation of $ m_{ \z, M }(\mathcal{I}) $.
When we consider a function with a sparse Fourier series however, any benefit in using an SFT to calculate the DFT of samples along the lattice is lost in comparison to the $ \mathcal{O}(d\, | \mathcal{I} |) $ size and operation count of the inverse computation.
For potentially large search spaces of multivariate frequencies $ \mathcal{I} $ such as a full cube $ \mathcal{I} = \left(\left(-\left \lceil \frac{N}{2} \right \rceil, \left \lfloor \frac{N}{2}  \right \rfloor\right] \cap \mathbbm{ Z }\right)^d $, both the time and memory complexity of this algorithm suffers from the curse of dimensionality.

The methods given in this paper instead work to use samples along possible larger lattices to produce sparse approximations of the Fourier transform of $ f $ without directly inverting $ m_{ \z, M } $.
The two algorithms considered below are able to operate on SFTs of manipulations of $ a $ in order to relate the univariate coefficients to their multivariate counterparts in $ o(| \mathcal{I} |) $-time.
This will allow the methods developed herein to run faster and with less memory than it takes to simply enumerate the frequency set $ \mathcal{I} $ and/or store $ m_{\z,M}( \mathcal{I} ) $ whenever $f$ has a sufficiently accurate sparse approximation.

\subsection{Prior work}%
\label{sub:prior_work}

Much recent work has considered the problem of quickly recovering both exactly sparse multivariate trigonometric polynomials as well as approximating more general functions by sparse trigonometric polynomials using dimension-incremental approaches \cite{volkmerdiss,Potts2016,Choi2020Sparse,Choi2019Sparse}.
These methods recover multivariate frequencies adaptively by searching for energetic frequencies on projections of the potential coefficient space $ \mathcal{I} \subset \left(\left(-\left \lceil \frac{N}{2} \right \rceil, \left \lfloor \frac{N}{2}  \right \rfloor\right] \cap \mathbbm{ Z }\right)^d $ into lower dimensional spaces.
These lower dimensional candidate sets are then paired together to build up a fully $ d $-dimensional search space smaller than the original one, which is expected to support the most energetic frequencies (see e.g., \cite[Section 3]{Kaemmerer2020} and the references within for a general overview).

In the context of Fourier methods, lattice-based techniques work well to handle support identification on the intermediary, lower-dimensional candidate sets, and especially recently, techniques based on multiple rank-1 lattices have shown success \cite{Kaemmerer2017,Kaemmerer2020}.
Though the total complexity in each of these steps is manageable and can be kept linear in the sparsity, these steps must be repeated in general to ensure that no potential frequencies have been left out.
In particular, this results in at least $ \mathcal{O}(ds^2N) $ operations (up to logarithmic factors) for functions supported on arbitrary frequency sets in order to obtain approximations that are guaranteed to be accurate with high probability.
Though from an implementational perspective, this runtime can be mitigated by completing many of the repetitions and initial one-dimensional searches in parallel, once pairing begins, the results of previous iterations must be synchronized and communicated to future steps, necessitating serial interruptions.

Other earlier works include \cite{Iwen13} in which previously existing univariate SFT results \cite{Iwen2010,Segal2013} were refined and adapted to the multivariate setting.
Though the resulting complexity on the dimension is well above the dimension-incremental approaches, deterministic guarantees are given for multivariate Fourier approximation in $ \mathcal{O}(d^4 s^2) $ (up to logarithmic factors) time and memory, as well as a random variant which dropped to linear scaling in $ s $, leading to a runtime on the order of $ \mathcal{O}(d^4 s) $ with respect to $ s $ and $ d $.
Additionally, the compressed sensing type guarantees in terms of Fourier compressibility of the function under consideration carry over from the univariate SFT analysis.
The scheme essentially makes use of a reconstructing rank-1 lattice on a superset of the full integer cube $ \mathcal{I} = \left(\left(-\left \lceil \frac{dN}{2} \right \rceil, \left \lfloor \frac{dN}{2}  \right \rfloor\right] \cap \mathbbm{ Z }\right)^d $ with certain number theoretic properties that allow for fast inversion of the resulting one-dimensional coefficients by the Chinese Remainder Theorem.
We note that this necessarily inflated frequency domain accounts for the suboptimal scaling in $ d $ above.

In \cite{Morotti2017}, another fully deterministic sampling strategy and reconstruction algorithm is given.
Like \cite{Iwen13} though, the method can only be applied to Fourier approximations over an ambient frequency space $ \mathcal{I} $ that is full $ d $-dimensional cube.
Moreover, the vector space structure exploited to construct the sampling sets necessitates that the side length $ N $ of this cube is the power of a prime.
However, the benefits to this construction are among the best considered so far: the method is entirely deterministic, has noise-robust recovery guarantees in terms of best $ s $-term estimates, the sampling sets used are on the order of $ \mathcal{O}(d^3 s^2 N) $, and the reconstruction algorithm's runtime complexity is on the order of $ \mathcal{O}(d^2 s^2 N^2) $ both up to logarithmic factors.
On the other hand, this algorithm still does not scale linearly in $ s $.

Finally, we discuss \cite{Choi2020High,Choi2019Multiscale}, a pair of papers detailing high-dimensional Fourier recovery algorithms which offer a simplified (and therefore faster) approach to lattice transforms and dimension-incremental methods.
These algorithms make heavy use of a one-dimensional SFT \cite{Christlieb2013,Christlieb2016} based on a phase modulation approach to discover energetic frequencies in a fashion similar to our Algorithm~\ref{alg:PhaseEnc} below.
The main idea is to recover entries of multivariate frequencies by using equispaced evaluations of the function along a coordinate axis as well as samples of the function at the same points slightly shifted (the remaining dimensions are generally ignored).
This shift in space produces a modulation in frequency from which frequency data can be recovered (cf.\ Lemma~\ref{lem:ShiftedFunctionCoefficients} and Algorithm~\ref{alg:PhaseEnc}).
By supplementing this approach with simple reconstructing rank-1 lattice analysis for repetitions of the full integer cube, the runtime and number of samples are given on average as $ \mathcal{O}(ds) $ up to logarithmic factors.

However, due to the possibility of collisions of multivariate frequencies under the hashing algorithms employed, these results hold only for random signal models.
In particular, theoretical results are only stated for functions with randomly generated Fourier coefficients on the unit circle with randomly chosen frequencies from a given frequency set.
Additionally, the analysis of these techniques assumes that the algorithm applied to the randomly generated signal does not encounter certain low probability (with respect to the random signal model considered therein) energetic frequency configurations.
Furthermore, the method is restricted in stability, allowing for spatial shifts in sampling bounded by at most the reciprocal of the side length of the multivariate frequency cube under consideration, and only exact recovery is considered (or recovery up to factors related to sample corruption by gaussian noise in \cite{Choi2019Multiscale}). In addition, no results given are proven concerning the approximation of more general periodic functions, e.g., compressible functions.

\subsection{Main contributions}%
\label{sub:main_contributions}

We begin with a brief summary of the benefits provided by our approach in comparison to the methods discussed above.
Below, we ignore logarithmic factors in our summary of the runtime/sampling complexities.
\begin{itemize}
	\item All variants, deterministic and random, of both algorithms presented in this paper have runtime and sampling complexities \textbf{linear in} $ d $ with best $ s $-term estimates for \textbf{arbitrary signals}.
		This is in contrast to the complexities of dimension-incremental approaches \cite{Choi2019Sparse,Choi2020Sparse,Kaemmerer2017,Kaemmerer2020} and the number theoretic approaches \cite{Iwen13,Morotti2017} while still achieving similarly strong best $ s $-term guarantees.
	\item Both algorithms proposed herein have randomized variants with runtime and sampling complexities \textbf{linear in} $ s $ with best $ s $-term estimates on \textbf{arbitrary signals} that hold \textbf{with high probability}.
		Thus, the randomized methods proposed in this paper achieve the efficient runtime complexities of \cite{Choi2020High,Choi2019Multiscale} while simultaneously exhibiting best $ s $-term approximation guarantees for general periodic functions thereby improving on the non-deterministic dimension incremental approaches \cite{Choi2019Sparse,Choi2020Sparse,Kaemmerer2017,Kaemmerer2020}.
	\item Both algorithms proposed herein have a deterministic variant with runtime and sampling complexities \textbf{quadratic} $ s $ with best $ s $-term estimates on \textbf{arbitrary signals} that also hold \textbf{deterministically}.
		This is in contrast to all previously discussed methods without deterministic guarantees, \cite{Choi2019Sparse,Choi2020Sparse,Kaemmerer2017,Kaemmerer2020,Choi2019Multiscale,Choi2020High}, as well as improving on prior deterministic results \cite{Iwen13,Morotti2017} for functions whose energetic frequency support sets $ \mathcal{I} $ are smaller than the full cube.
\end{itemize}

\subsubsection*{Overview of the methods and related theory}%
\label{ssub:overview_of_the_methods_and_related_theory}

We will build on the structure of the fast and potentially deterministic one-dimensional SFTs from \cite{Iwen13} and its discrete variant from \cite{Mehri2019} by applying those techniques along rank-1 lattices.
As previously discussed, the primary difficulty in doing so is determining a way to extract the desired multidimensional frequency information for those most energetic one-dimensional frequencies identified in an efficient and provably accurate way.
We propose and analyze two different methods for solving this problem herein.

The first frequency identification approach involves modifications of the phase shifting from \cite{Christlieb2013,Christlieb2016,Choi2020High,Choi2019Multiscale} in Algorithm~\ref{alg:PhaseEnc}.
By employing the phase shifting process from these works in conjunction with SFTs over an arbitrary reconstructing rank-1 lattice for our multivariate frequency search space $ \mathcal{I} $, we achieve a new class of fast method with several new benefits.
Notably, we are able to maintain error guarantees for any function (not just random signals) in terms of best Fourier $ s $-term approximations.
Additionally, we factor the instability and potential for collisions from \cite{Choi2020High,Choi2019Multiscale} into these best $ s $-term approximations, suffering only a linear factor of $ N $ from the more typical results produced by our second technique discussed in the next paragraph (cf.\ Corollaries \ref{cor:PhaseEncRecoveryDiscreteSFT} and \ref{cor:PhaseEncRecoverySublinearSFT}).
Finally, we are able to maintain quadratic in $ s $/deterministic and linear in $ s $/random runtime and sampling complexities while reducing the dependence on the dimension of the function's domain $ d $ from \cite{Iwen13} to a linear one (cf.\ Lemma \ref{lem:PhaseEncRecovery}).

Our second technique in Algorithm~\ref{alg:2dDFT} uses a more novel approach to applying SFTs to modifications of the multivariate function along a reconstructing rank-1 lattice.
By using a potentially larger rank-1 lattice than one that suffices only for being reconstructing on $ \mathcal{I} $, we restrict the function to only $ d - 1 $ dimensions of the lattice at a time, allowing one dimension to remain free.
Applying SFTs along the lattice constrained variables, FFTs in the free dimension, and synchronizing based on known Fourier coefficients (for example, from simply applying an SFT on the entire function restricted to the lattice) allows one to reconstruct the full multivariate coefficients with fewer possibilities for numerical instability.
In particular this produces more accurate best Fourier $ s $-term approximation guarantees (cf.\ Corollaries~\ref{cor:2dDFTRecoveryDiscreteSFT} and \ref{cor:2dDFTRecoverySublinearSFT}).
We again maintain the linear in $ d $, quadratic in $ s $/deterministic, and linear in $ s $/random sampling and runtime complexities, however, we now additionally incur a penalty of a quadratic factor of $ N $ (cf.\ Lemma~\ref{lem:2dDFTRecovery}).

We stress here that by compartmentalizing the translation from multivariate analysis to univariate analysis in Algorithms~\ref{alg:PhaseEnc} and \ref{alg:2dDFT} into the theory of rank-1 lattices, we additionally abstract our interaction with the multivariate frequency domain of interest.
As such, our techniques are suitable for any arbitrary frequency set of interest $ \mathcal{I} $ constrained only by our necessity for a reconstructing rank-1 lattice for $ \mathcal{I} $ (and potentially projections of $ \mathcal{I} $ in the case of Algorithm~\ref{alg:2dDFT}).
This flexibility allows our methods to supersede the results from \cite{Iwen13}, primarily with respect to the polynomial factor of $ d $ in our runtime and sampling complexities.
We remark that though the existence of the necessary reconstructing rank-1 lattice is a nontrivial requirement, there exist efficient construction algorithms for arbitrary frequency sets via deterministic component by component methods, see e.g., \cite{Kaemmerer2014,Kuo2020,Plonka2018}.

We also emphasize that the multivariate structures we employ are also entirely modular with respect to their underlying univariate components. 
More specifically, they can make use of any reasonably fast and theoretically sound SFT approach to produce resulting runtime and sampling estimates which scale well in the complexity of the underlying SFT algorithm (by only a factor of $ d $ in Algorithm~\ref{alg:PhaseEnc} and a factor of $ dN $ in Algorithm~\ref{alg:2dDFT}).
Lemmas~\ref{lem:PhaseEncRecovery} and \ref{lem:2dDFTRecovery} are therefore presented in a similarly modular fashion.
To provide specific recovery results we therefore use the univariate SFTs from \cite{Iwen13,Mehri2019}.
Notably, we also improve the theoretical approximation guarantees of these univariate SFTs in their own right and in the process include the addition of a robust variant of the discrete method in \cite{Mehri2019}.

Finally, the methods we present are trivially parallelizable so that in particular, a large majority of the additional factors of $ d $ or $ dN $ respectively in the runtimes of Algorithm~\ref{alg:PhaseEnc} or Algorithm~\ref{alg:2dDFT} discussed above can occur in parallel.

\subsection{Organization}%
\label{sub:organization}

The remainder of this paper is presented as follows: in Section~\ref{sec:notation-and-assumptions}, we set the notation, the notions of the Fourier transform, and the various types of manipulations we will be using in the sequel.
Section~\ref{sec:1dSFTs} reviews and further refines the univariate SFTs from \cite{Iwen13,Mehri2019} to suit our multivariate analysis.
Section~\ref{sec:multivariate-results} presents our main multivariate approximation algorithms and their analysis.
In particular, Section~\ref{sub:PhaseEnc} discusses the phase-shifting approach, while Section~\ref{sub:2dDFT} discusses the two-dimensional SFT/DFT combination approach.
Finally, we implement these two algorithms numerically and present the empirical results in Section~\ref{sec:numerics}.

\section{Notation and assumptions}
\label{sec:notation-and-assumptions}

\subsection{Multivariate}

We begin by defining a one-dimensional frequency band of length $ N $ as $ \mathcal{B}_N := \left(- \left \lceil \frac{N}{2} \right \rceil, \left \lfloor \frac{N}{2} \right \rfloor \right] \cap \mathbbm{ Z }$.
For a potentially large but finite multivariate frequency set $ \mathcal{I} $, which we think of as containing the most significant frequencies of the function under consideration, we choose $ N = \max_{ \ell \in [d] }(\max_{ \k \in \mathcal{I}} k_\ell - \min_{ \tilde \k \in \mathcal{I} } \tilde k_\ell ) + 1$ as the minimal width such that $ \mathcal{I} \subset \h + \mathcal{B}_N^d $ for some $ \h \in \mathbbm{ Z }^d $.
By appropriately modulating any multivariate function $ f : \mathbbm{ T }^d \rightarrow \C $ under consideration, i.e., considering $ \mathbbm{ e }^{ -2 \pi \mathbbm{ i } \h \cdot \circ } f $, we shift the frequencies of Fourier coefficients of $ f $ originally in $ \mathcal{I} $ to $ \mathcal{I} - \h \subset \mathcal{B}_N^d $.
Thus, we assume without loss of generality that $ \mathcal{I} \subset \mathcal{B}_N^d $ with $ N $ as above.
Without loss of generality, we will also assume that for a reconstructing rank-1 lattice $ \Lambda(\z, M) $, the generating vector satisfies $ \z \in [M]^d $.

To avoid confusion with the hat notation which will be reserved for univariate functions below, we denote the sequence of all Fourier coefficients (i.e., the Fourier transform) of a periodic function $ f: \mathbbm{ T }^d \rightarrow \C $ as $ c(f) = (c_\k(f))_{ \k \in \mathbbm{Z}^d } $, also writing this as simply $ c $ when the function is clear from context.
Its restriction to $ \mathcal{I} $ is denoted $ c(f)\rvert_{ \mathcal{I} } = (c_\k(f))_{ \k \in \mathcal{I} } $, and the best $ s $-term approximation, that is, its restriction to the support of the $ s $-largest magnitude entries, is denoted $ c_s^\mathrm{opt} $ or $ (c\rvert_{ \mathcal{I} })_s^\mathrm{opt} $ on $ \mathbbm{Z}^d $ or $ \mathcal{I} $, respectively.
We denote multiindexed vectors only defined on finite index sets (which are not restrictions of infinitely indexed sequences) in boldface, e.g., $ \mathbf{b} = (b_{ \k })_{ \k \in \mathcal{I} } $, as well as identify this multivariate vector as a one-dimensional vector $ \mathbf{b} \in \C^{ | \mathcal{I} |} $ via lexicographic ordering when dictated by context.
Again dictated by context, we also extend these multiindexed vectors to larger index sets by setting them to zero outside of their original domain.
For example, if $ \mathbf{b} = (b_\k)_{ \k \in \mathcal{I} } $ and $ c = (c_\k)_{ \k \in \mathbbm{ Z }^d } $,
\begin{equation*}
	\| \mathbf{b} - c\|_{ \ell^1( \mathbbm{ Z }^d) } = \sum_{ \k \in \mathcal{I} } |b_\k - c_\k| + \sum_{ \k \in \mathbbm{ Z }^d \setminus \mathcal{I} } |c_\k|.
\end{equation*}

In the multivariate approaches which follow, we will also make use of the shift operator $ S_{ \ell, \alpha } $ for dimension $ \ell \in [d] $ and shift $ \alpha \in \mathbbm{ R } $ defined by its action on the multivariate periodic function $ f: \mathbbm{ T }^d \rightarrow \C $ as
\begin{equation*}
	S_{ \ell, \alpha }(f)(x_1, \ldots, x_d) := f(x_1, \ldots, x_{ \ell - 1 }, (x_\ell + \alpha) \bmod 1, x_{ \ell + 1 }, \ldots, x_d).
\end{equation*}
When necessary, we will separate out coordinate $ \ell $ of a multivariate point $ \x \in \mathbbm{ T }^d $ or frequency $ \k \in \mathbbm{ Z }^d $, denoting the remaining coordinates as $ \x_\ell' \in \mathbbm{ T }^{ d - 1 }$ or $ \k_\ell' \in \mathbbm{ Z }^{ d - 1 } $.
With a slight abuse of notation, we can rewrite the original point or frequency as $ \x = (x_\ell, \x_\ell') $  or $ \k = (k_\ell, \k_\ell') $.

\subsection{Univariate}

For any univariate periodic function $ a: \mathbbm{ T } \rightarrow \C $, we define the vector $ \a \in \C^M $ as the vector of $ M $ equispaced samples of $ a $ on $ \mathbbm{ T } $, that is, $ \a = ( a( j/M ))_{ j \in [M] } $.
As in the multivariate case, we define the Fourier transform of $ a: \mathbbm{ T } \rightarrow \C $ as the sequence $ \hat a = ( \hat a_\omega )_{ \omega \in \mathbbm{ Z } } $ with
\begin{equation*}
	\hat a_\omega := \int_{ \mathbbm{ T } } a(t) \, \mathbbm{ e }^{ -2\pi \mathbbm{ i } \omega x } \, \mathrm{d} x \text{ for all } \omega \in \mathbbm{ Z }.
\end{equation*}
Additionally, we define the vector $ \hat \a \in \C^M $ as the restriction of $ \hat a $ to $ \mathcal{B}_M $.
If not explicitly stated, the length of the discretized function $ \a $ and Fourier transform $ \hat \a $ will be clear from context.
Note that $ \hat \a $ is not necessarily the discrete Fourier transform of $ \a $, which we define as
\begin{align*}
	(\F_M \, \a)_\omega &:= \frac{1}{M} \sum_{ j \in [M] } a_j \, \mathbbm{ e }^{ -2 \pi \mathbbm{ i } \omega j / M } = \frac{1}{M} \sum_{ j \in [M] } a \left( \frac{j}{M} \right) \mathbbm{ e }^{ -2 \pi \mathbbm{ i } \omega j / M }, \text{ where} \\
    \F_M &:= \left( \mathbbm{ e }^{ -2 \pi \mathbbm{ i } \omega j / M } / M \right)_{ j \in [M],\, \omega \in \mathcal{B}_M }
\end{align*}
is the discrete Fourier matrix.
Our convention here and in the remainder of the paper is to use zero-based indexing which is always taken implicitly modulo the length of the dimension, e.g., $ (\F_M)_{ 0, -1 } = (\F_M)_{ 0, M - 1 } $.

For any vector $ \mathbf{ b } \in \C^M $, we denote its best $ s $-term approximation $ \mathbf{ b }_{ s }^\mathrm{opt} $, where as above, $ \mathbf{ b }_s^\mathrm{opt} $ is the restriction of $ \mathbf{b} $ to its $ s $ largest magnitude entries.
In the sequel, we always assume that our sparsity parameters $ s $ are at most half the size of the vectors under consideration so that, e.g., $ \mathbf{b}_{ 2s }^\mathrm{opt} $ is well-defined.
Additionally as above, vectors can also be compared with other vectors on larger index sets than they are defined by simply setting the smaller vectors to zero outside of their original domain.

As for one-dimensional approximations, we will be considering SFT algorithms which, given sparsity parameter $ s $ and bandwidth $ M $, produce an $ s $-sparse approximation of the Fourier transform of a function $ a \in C( \mathbbm{ T } ) $ restricted to $ \mathcal{B}_M $.
Note that these are not necessarily discrete algorithms which take in $ \a $ as input.
We denote these algorithms $ \mathcal{A}_{ s, M}: C( \mathbbm{ T } ) \rightarrow \C^M $, which produce $ \mathcal{A}_{ s, M } a =: \v \in \C^M $ as approximations to $ \hat \a \in \C^M $ using some fixed number of samples of $ a $.

\section{One-dimensional sparse Fourier transform results}
\label{sec:1dSFTs}

Below, we summarize some of the previous work on one-dimensional sparse Fourier transforms which will be used in our multivariate algorithms.
Note that we will consider algorithms which produce $ 2s $-sparse approximations of the Fourier coefficients of a given signal and satisfy error guarantees in terms of the best $ 2s $ \emph{and} $ s $-term approximations.
We first review the sublinear-time algorithm from \cite{Iwen13} which uses fewer than $ M $ nonequispaced samples of a function.
Below, we will present slightly improved error bounds which necessitate the following lemma.

\begin{lem}
	\label{lem:ErrorInTauSigFrequencies}
	For $ \x \in \C^K $ and $ \mathcal{S}_\tau := \{ k \in [K] \mid |x_k| \geq \tau\} $, if $ \tau \geq \frac{\|\x - \x_s^\mathrm{opt}\|_1}{s} $, then $ | \mathcal{S}_\tau | \leq 2s $ and
	\begin{equation*}
		\|\x - \x\rvert_{ \mathcal{S}_\tau }\|_{ 2 } \leq \|\x - \x_{ 2s }^\mathrm{opt} \|_2 + \tau \sqrt{ 2s } .
	\end{equation*}
\end{lem}
\begin{proof}
	Ordering the entries of $ \x $ in descending order (with ties broken arbitrarily) as $ |x_{ k_1 }| \geq |x_{ k_2 }| \geq \ldots $, we first note that
	\begin{equation*}
		\|\x - \x_s^\mathrm{opt}\|_1 \geq \sum_{ j = s + 1 }^{ 2s } |x_{ k_j }| \geq s |x_{ k_{ 2s } }|.
	\end{equation*}
	By assumption then, $ \tau \geq | x_{ k_{ 2s } }| $, and since $ \mathcal{S}_\tau$ contains the $ | \mathcal{S}_\tau| $-many largest entries of $ \x $, we must have $ \mathcal{S}_\tau \subset \supp (\x_{ 2s }^\mathrm{opt}) $.
	Note then that $ | \mathcal{S}_\tau | \leq 2s $.
	Finally, we calculate
	\begin{align*}
		\|\x - \x\rvert_{ \mathcal{S}_\tau }\|_2 
			&\leq \|\x - \x_{ 2s }^\mathrm{opt} \|_2 + \|\x_{ 2s }^\mathrm{opt} - \x_{ \mathcal{S}_\tau }\|_2 \\
			&\leq \|\x - \x_{ 2s }^\mathrm{opt} \|_2 + \sqrt{ \sum_{ k \in \supp (\x_{ 2s }^\mathrm{opt}) \setminus \mathcal{S}_\tau } x_k^2 } \\
			&\leq \|\x - \x_{ 2s }^\mathrm{opt} \|_2 + \tau \sqrt{ 2s }
	\end{align*}
	completing the proof.
\end{proof}

\begin{thm}[Robust sublinear-time, nonequispaced SFT: \cite{Iwen13}, Theorem 7/\cite{Mehri2019}, Lemma 4]
	\label{thm:1dSFT}
	For a signal $ a \in \mathcal{W}(\mathbbm{ T }) \cap C( \mathbbm{ T } ) $ corrupted by some arbitrary noise $ \mu: \mathbbm{ T } \rightarrow \C $, Algorithm 3 of \cite{Iwen13}, denoted $ \mathcal{A}_{ 2s, M }^{ \mathrm{sub} } $, will output a $ 2s $-sparse coefficient vector $ \v \in \C^M $ which
	\begin{enumerate}
		\item \label{it:1dSFT:FindSignificantFrequencies}
			reconstructs every frequency of $ \hat \a \in \C^M $, $ \omega \in \mathcal{B}_M $, with corresponding Fourier coefficients meeting the tolerance
			\begin{equation*}
				| \hat a_\omega | > (4 + 2 \sqrt{ 2 }) \left( \frac{\| \hat \a - \hat \a_{ s }^\mathrm{opt}\|_1}{s} + \| \hat a - \hat \a \|_1 + \| \mu \|_\infty \right),
			\end{equation*}
		\item \label{it:1dSFT:InftyEstimate} satisfies the $ \ell^\infty $ error estimate for recovered coefficients
			\begin{equation*}
				\| (\hat \a - \v)\rvert_{ \supp (\v) } \|_\infty \leq \sqrt{ 2 } \left( \frac{\| \hat \a - \hat \a_{ s }^\mathrm{opt}\|_1}{s} + \| \hat a - \hat \a \|_1 + \| \mu \|_\infty \right),
			\end{equation*}
		\item \label{it:1dSFT:2Estimate} satisfies the $ \ell^2 $ error estimate
			\begin{equation*}
				\| \hat \a - \v \|_2 \leq \| \hat \a - \hat \a_{ 2s }^\mathrm{opt} \|_2 + \frac{(8 \sqrt{ 2 } + 6)\| \hat \a - \hat \a_s^\mathrm{opt} \|_1}{\sqrt{ s }} + (8 \sqrt{ 2 } + 6)\sqrt{ s } (\| \hat a - \hat \a \|_1 + \| \mu \|_\infty ),
			\end{equation*}
		\item \label{it:1dSFT:Complexity} and the number of required samples of $ a $ and the operation count for $ \mathcal{A}_{ 2s, M }^{ \mathrm{sub} } $ are
			\begin{equation*}
				\mathcal{O} \left( \frac{s^2\log^4 M}{\log s}  \right).
			\end{equation*}
	\end{enumerate}
	The Monte Carlo variant of $ \mathcal{A}_{ 2s, M }^\mathrm{sub} $, denoted $ \mathcal{A}_{ 2s, M }^\mathrm{sub, MC} $, referred to by Corollary 4 of \cite{Iwen13} satisfies all of the conditions (1) -- (3) simultaneously with probability $ (1 - \sigma) \in [2/3, 1) $ and has number of required samples and operation count
	\begin{equation*}
		\mathcal{O} \left( s \log^3(M) \log \left( \frac{M}{\sigma} \right) \right).
	\end{equation*}
	The samples required by $ \mathcal{A}_{ 2s, M }^\mathrm{sub, MC} $ are a subset of those required by $ \mathcal{A}_{ 2s, M }^\mathrm{sub} $.
\end{thm}
\begin{proof}
	We refer to \cite[Theorem 7]{Iwen13} and its modification for noise robustness in \cite[Lemma 4]{Mehri2019} for the proofs of properties \eqref{it:1dSFT:InftyEstimate} and \eqref{it:1dSFT:Complexity}.
	As for \eqref{it:1dSFT:FindSignificantFrequencies}, \cite[Lemma 6]{Iwen13} and its modification in \cite[Lemma 4]{Mehri2019} imply that any $ \omega \in \mathcal{B}_M $ with $ |\hat a_\omega| > 4 (\|\hat \a - \hat \a_s^\mathrm{opt}\|_1 / s + \|\hat a - \hat \a \|_1 + \| \mu \|_\infty) =: 4\delta $ will be identified in \cite[Algorithm 3]{Iwen13}. 
	An approximate Fourier coefficient for these and any other recovered frequencies is stored in the vector $ \x $ which satisfies the same estimate in property \eqref{it:1dSFT:InftyEstimate} by the proof of \cite[Theorem 7]{Iwen13} and \cite[Lemma 4]{Mehri2019}.
	However, only the $ 2s $ largest magnitude values of $ \x $ will be returned in $ \v $.
	We therefore analyze what happens when some of the potentially large Fourier coefficients corresponding to frequencies in $ \mathcal{S}_{ 4\delta } $ do not have their approximations assigned to $ \v $.

	For the definition of $ \mathcal{S}_\tau $ in Lemma~\ref{lem:ErrorInTauSigFrequencies} applied to $ \hat \a $, we must have $ | \mathcal{S}_{ 4\delta }| \leq 2s = |\supp(\v)|$.
	If $ \omega \in \mathcal{S}_{ 4\delta } \setminus \supp(\v) $, there must then exist some other $ \omega' \in \supp(\v) \setminus \mathcal{S}_{ 4\delta } $ which was identified and took the place of $ \omega $ in $ \supp(\v) $.
	For this to happen, $ |\hat a_{ \omega' }| \leq 4\delta $ and $ |x_{ \omega' }| \geq |x_{ \omega }| $.
	But by property \eqref{it:1dSFT:InftyEstimate} (extended to all coefficients in $ \x $), we know
	\begin{equation*}
		4 \delta + \sqrt{ 2 } \delta \geq |\hat a_{ \omega' }| + \sqrt{ 2 } \delta \geq |x_{ \omega' }| \geq |x_\omega| \geq |\hat a_\omega| - \sqrt{ 2 }\delta.
	\end{equation*}
	Thus, any frequency in $ \mathcal{S}_{ 4\delta } $ not chosen satisfies $ |\hat a_\omega| \leq (4 + 2\sqrt{ 2 }) \delta $, and so every frequency in $ \mathcal{S}_{ (4 + 2\sqrt{ 2 }) \delta } $ is in fact identified in $ \v $ verifying property \eqref{it:1dSFT:FindSignificantFrequencies}.

	As for property \eqref{it:1dSFT:2Estimate}, we estimate the $ \ell^2 $ error using property \eqref{it:1dSFT:InftyEstimate}, Lemma~\ref{lem:ErrorInTauSigFrequencies}, and the above argument as
	\begin{align*}
		\| \hat \a - \v \|_2 
			&\leq \| \hat \a - \hat \a\rvert_{ \supp(\v) } \|_2 + \|(\hat \a - \v )\rvert_{ \supp(\v) } \|_2 \\
			&\leq \| \hat \a - \hat \a\rvert_{ \mathcal{S}_{ 4\delta } \cap \supp (\v) } \|_2 + \sqrt{ 2 } \delta \sqrt{ 2 s }\\
			&\leq \| \hat \a - \hat \a\rvert_{ \mathcal{S}_{ 4\delta } } \|_2 + \| \hat \a\rvert_{ \mathcal{S}_{ 4\delta } \setminus \supp(\v) }\|_2 + 2 \delta \sqrt{ s } \\
			&\leq \| \hat \a - \hat \a_{ 2s }^\mathrm{opt} \|_2 + 4 \delta \sqrt{ 2s } + (4 + 2\sqrt{ 2 }) \delta \sqrt{ 2s } + 2 \delta \sqrt{ s }\\
			&= \| \hat \a - \hat \a_{ 2s }^\mathrm{opt} \|_2 + (8 \sqrt{ 2 } + 6) \sqrt{ s } \delta
	\end{align*}
	as desired.
\end{proof}

\begin{rem}
	In the noiseless case, if the univariate function~$ a $ is Fourier $ s $-sparse, i.e., is a trigonometric polynomial and $ M $ is large enough such that $ \supp (\hat a) \subset \mathcal{B}_M $, both $ \mathcal{A}_{ 2s, M }^\mathrm{sub} $ and $ \mathcal{A}_{ 2s, M }^\mathrm{sub, MC} $ will exactly recover $ \hat \a $ (the latter with probability $ 1 - \sigma$), and therefore $ \hat a $.
	In particular, note that the output of either algorithm will then actually be $ s $-sparse.
\end{rem}

Using the above SFT algorithm with the discretization process outlined in \cite{Mehri2019} leads to a fully \emph{discrete} sparse Fourier transform, requiring only equispaced samples of $ a $.
However, rather than separately accounting for the truncation to the frequency band $ \mathcal{B}_M $ as above, the equispaced samples allow us to take advantage of aliasing, which is particularly important when we apply the algorithm along reconstructing rank-1 lattices.
Thus, instead of approximating $ \hat \a \in \C^M $, the restriction of $ \hat a $ to $ \mathcal{B}_M $, as above, we prefer to approximate the discrete Fourier transform of $ \a $.
We now review how these notions of restrictions versus aliasing interact.

\begin{lem}
	\label{lem:DFTEqualsFT}
	Let $ a \in C( \mathbbm{ T } ) $ be bandlimited with $ \supp (\hat a) \subset \mathcal{B}_M $.
	Then $ \hat \a = \F_M \, \a  $.
\end{lem}
\begin{proof}
	Writing $ a(t) = \sum_{ \omega \in \mathcal{B}_M } \hat a_\omega \, \mathbbm{ e }^{ 2 \pi \mathbbm{ i } \omega t } $, for any $ \omega \in \mathcal{B}_M $, we calculate
	\begin{align*}
		(\F_M \, \a)_\omega 
            &= \frac{1}{M} \sum_{ j \in \mathcal{B}_M } a\left(\frac{j}{M}\right) \mathbbm{ e }^{ - 2 \pi \mathbbm{ i } \omega j / M }
			= \frac{1}{M} \sum_{ j \in \mathcal{B}_M } \left( \sum_{ \tilde \omega \in \mathcal{B}_M } \hat a_{ \tilde \omega } \, \mathbbm{ e }^{ 2 \pi \mathbbm{ i } \tilde \omega j / M } \right) \mathbbm{ e }^{ - 2 \pi \mathbbm{ i } \omega j / M } \\
			&= \frac{1}{M} \sum_{ j \in \mathcal{B}_M } \sum_{ \tilde \omega \in \mathcal{B}_M } \hat a_{ \tilde \omega } \, \mathbbm{ e }^{ 2 \pi \mathbbm{ i } (\tilde \omega - \omega) j / M }
			= \sum_{ \tilde \omega \in \mathcal{B}_M } \hat a_{ \tilde \omega } \, \delta_{ 0, (\tilde \omega - \omega \bmod M) } \\
			&= \hat a_{ \omega },
	\end{align*}
	as desired.
\end{proof}

\begin{lem}
	\label{lem:DFTAliasing}
	For any function $ a: \mathbbm{ T } \rightarrow \C $ with Fourier series $ a(t) = \sum_{ \omega \in \mathbbm{ Z } } \hat a_\omega \, \mathbbm{ e }^{ 2 \pi \mathbbm{ i } \omega t } $, define the aliased polynomial
	\begin{equation*}
		a_\mathrm{alias}(t) = \sum_{ \omega \in \mathcal{B}_M } \underbrace{\left( \sum_{ \tilde \omega \equiv \omega \imod{M} } \hat a_{ \tilde \omega } \right)}_{=:(\hat \a_\mathrm{alias})_{\omega}} \mathbbm{ e }^{ 2 \pi \mathbbm{ i } \omega t }.
	\end{equation*}
	Then the equispaced samples coincide, giving $ \a = \a_\mathrm{alias} \in \C^M $ and $ \hat \a_\mathrm{alias} = \F_M \, \a $.
\end{lem}
\begin{proof}
	We group frequencies in the Fourier series of $ a $ by their residues in $ \mathcal{B}_M $, giving
	\begin{align*}
		\left( \a \right)_j
			&= \sum_{ \tilde \omega \in \mathbbm{ Z } } \hat a_{ \tilde \omega } \, \mathbbm{ e }^{ 2 \pi \mathbbm{ i } \tilde \omega j / M }
			= \sum_{ \omega \in \mathcal{B}_M } \sum_{ n \in \mathbbm{ Z } } \hat a_{ \omega + nM } \, \mathbbm{ e }^{ 2 \pi \mathbbm{ i } (\omega + nM)j / M }
			= \sum_{ \omega \in \mathcal{B}_M } \left( \sum_{ \tilde \omega \equiv \omega \imod{M} } \hat a_{ \tilde \omega } \right) \mathbbm{ e }^{ 2 \pi \mathbbm{ i } \omega j / M } \\
			&= \left( \a_\mathrm{alias} \right)_j \text{ for all } j \in [M].
	\end{align*}
	Now, since $ \supp (\hat \a_\mathrm{alias}) \subset \mathcal{B}_M $, Lemma~\ref{lem:DFTEqualsFT} implies $ \hat \a_\mathrm{alias} = \F_M \, \a_\mathrm{alias} = \F_M \, \a $.
\end{proof}

Eventually, we will consider techniques for approximation of arbitrary periodic functions rather than simply polynomials.
For this reason, we require noise-robust recovery results for the method in \cite{Mehri2019}.
The necessary modifications to account for this robustness as well as the improved guarantees carried over from the previous algorithm are given below.

\begin{thm}[Robust discrete sublinear-time  SFT: see \cite{Mehri2019}, Theorem 5]
	\label{thm:1dDiscreteSFT}
	For a signal $ a \in \mathcal{W}(\mathbbm{ T }) \cap C( \mathbbm{ T } ) $ corrupted by some arbitrary noise $ \mu: \mathbbm{ T } \rightarrow \C $, and $ 1 \leq r \leq \frac{M}{36} $ Algorithm 1 of \cite{Mehri2019}, denoted $ \mathcal{A}_{ 2s, M }^{ \mathrm{disc} } $, will output a $ 2s $-sparse coefficient vector $ \v \in \C^M $ which
	\begin{enumerate}
		\item \label{it:1dDiscreteSFT:FindSignificantFrequencies} reconstructs every frequency of $ \F_M \, \a \in \C^M $, $ \omega \in \mathcal{B}_M $, with corresponding aliased Fourier coefficient meeting the tolerance
			\begin{equation*}
				| (\F_M \, \a)_\omega | > 12(1 + \sqrt{ 2 }) \left( \frac{\| \F_M \, \a - (\F_M \, \a)_s^\mathrm{opt}\|_1}{2s} + 2( \| \a \|_\infty M^{ -r } + \| \boldsymbol{\mu} \|_{ \infty } )\right),
			\end{equation*}
		\item \label{it:1dDiscreteSFT:InftyEstimate} satisfies the $ \ell^\infty $ error estimate for recovered coefficients
			\begin{equation*}
				\| (\F_M \, \a - \v)\rvert_{ \supp (\v) } \|_\infty \leq 3 \sqrt{ 2 } \left( \frac{\| \F_M \, \a- (\F_M \, \a)_s^\mathrm{ opt }\|_1}{2s} + 2 (\| \a \|_\infty M^{ -r } + \| \boldsymbol{\mu} \|_\infty)  \right),
			\end{equation*}
		\item \label{it:1dDiscreteSFT:2Estimate} satisfies the $ \ell^2 $ error estimate
			\begin{equation*}
				\| \F_M \, \a - \v \|_2 \leq \| \F_M \, \a - (\F_M \, \a)_{ 2s }^\mathrm{opt} \|_2 + 38 \frac{\| \F_M \, \a - (\F_M \, \a)_s^\mathrm{opt} \|_1}{\sqrt{ s }} + 152 \sqrt{ s } (\| \a \|_\infty M^{ -r } +  \| \boldsymbol{\mu} \|_\infty  ),
			\end{equation*}
		\item \label{it:1dDiscreteSFT:Complexity} and the number of required samples of $ \a $ and the operation count for $ \mathcal{A}_{ 2s, M }^\mathrm{disc} $ are
			\begin{equation*}
				\mathcal{O} \left( \frac{s^2 r^{ 3/2 } \log^{ 11/2 } M}{\log s}  \right).
			\end{equation*}
	\end{enumerate}
	The Monte Carlo variant of $ \mathcal{A}_{ 2s, M }^\mathrm{disc} $, denoted $ \mathcal{A}_{ 2s, M }^\mathrm{disc, MC} $, satisfies the all of the conditions (1) -- (3) simultaneously with probability $ (1 - \sigma) \in [2/3, 1) $ and has number of required samples and operation count
	\begin{equation*}
		\mathcal{O} \left( s r^{ 3/2 }\log^{ 9/2 }(M) \log \left( \frac{M}{\sigma} \right) \right).
	\end{equation*}
\end{thm}
\begin{proof}
	All notation in this proof matches that in \cite{Mehri2019} (in particular, we use $ f $ to denote the one-dimensional function in place of $ a $ in the theorem statement and $ N = 2M + 1$).
	We begin by substituting the $ 2\pi $-periodic gaussian filter given in (3) on page 756 with the $ 1 $-periodic gaussian and associated Fourier transform
	\begin{equation*}
		g(x) = \frac{1}{c_1} \sum_{ n = -\infty }^\infty \mathbbm{ e }^{ - \frac{(2\pi)^2(x - n)^2}{2 c_1^2}  },\quad
		\hat g_\omega = \frac{1}{\sqrt{ 2\pi }} \mathbbm{ e }^{ - \frac{c_1^2 \omega^2}{2} }.
	\end{equation*}
	Note then that all results regarding the Fourier transform remain unchanged, and since this $ 1 $-periodic gaussian is a just a rescaling of the $ 2\pi $-periodic one used in \cite{Mehri2019}, the bound in \cite[Lemma 1]{Mehri2019} holds with a similarly compressed gaussian, that is, for all $ x \in \left[ -\frac{1}{2}, \frac{1}{2} \right] $
	\begin{equation}
		\label{eq:GaussianBound}
		g(x) \leq \left( \frac{3}{c_1} + \frac{1}{\sqrt{ 2\pi }} \right) \mathbbm{ e }^{ -\frac{ \left( 2\pi x \right)^2 }{2 c_1^2} }.
	\end{equation}
	Analogous results up to and including \cite[Lemma 10]{Mehri2019} for $ 1 $-periodic functions then hold straightforwardly.

	Assuming that our signal measurements $ \f = (f(y_j))_{ j = 0 }^{ 2M } = (f( \frac{j}{N} ))_{ j = 0 }^{ 2M } $ are corrupted by some discrete noise $ \boldsymbol{\mu} = (\mu_j)_{ j = 0 }^{ 2M } $, we consider for any $ x \in \mathbbm{T} $ a similar bound to \cite[Lemma 10]{Mehri2019}.
	Here, $ j' := \arg \min_{ j }| x - y_j| $ and $ \kappa := \left \lceil \gamma \ln N \right \rceil + 1 $ for some $ \gamma \in \mathbbm{ R }^+ $ to be determined.
	Then,
	\begin{align*}
		&\left| \frac{1}{N} \sum_{ j = 0 }^{ 2M } f(y_j) g(x - y_j) - \frac{1}{N}  \sum_{ j = j' - \kappa }^{ j' + \kappa } (f(y_j) + \mu_j) g(x - y_j) \right| \\
		&\qquad \leq \frac{1}{N} \left| \sum_{ j = 0 }^{ 2M } f(y_j) g( x-y_j ) - \sum_{ j = j' - \kappa}^{ j'+\kappa } f(y_j) g(x - y_j) \right| + \frac{1}{N} \left| \sum_{ j = j' - \kappa }^{ j' + \kappa } \mu_j g(x - y_j) \right| \\
		&\qquad \leq \frac{1}{N} \left| \sum_{ j = 0 }^{ 2M } f(y_j) g( x-y_j ) - \sum_{ j = j' - \kappa}^{ j'+\kappa } f(y_j) g(x - y_j) \right| + \frac{1}{N} \| \boldsymbol{\mu} \|_\infty \sum_{ k = - \kappa }^{ \kappa } g(x - y_{ j' + k })
	\end{align*}
	We bound the first term in this sum by a direct application of \cite[Lemma 10]{Mehri2019}; however, we take this opportunity to reduce the constant in the bound given there.
	In particular, bounding this term by the final expression in the proof of \cite[Lemma 10]{Mehri2019} and using our implicit assumption that $ 36 \leq N $, we have
	\begin{equation}
		\label{eq:NoisyConvolutionBound}
		\begin{aligned}
			&\left| \frac{1}{N} \sum_{ j = 0 }^{ 2M } f(y_j) g(x - y_j) - \frac{1}{N}  \sum_{ j = j' - \kappa }^{ j' + \kappa } (f(y_j) + \mu_j) g(x - y_j) \right| \\
			&\qquad\qquad \leq \left( \frac{3}{\sqrt{ 2\pi }} + \frac{1}{2\pi} \sqrt{ \frac{\ln 36}{36} } \right) \| \mathbf{f} \|_\infty N^{ -r } + \frac{1}{N} \| \boldsymbol{\mu} \|_\infty \sum_{ k = - \kappa }^{ \kappa } g(x - y_{ j' + k }).
		\end{aligned}
	\end{equation}

	We now work on bounding the second term.
	First note that for all $ k \in [-\kappa, \kappa] \cap \mathbbm{ Z } $,
	\begin{align*}
		g(x-y_{ j' \pm k }) 
			&= g\left(x - y_{ j' } \pm \frac{k}{N}\right).
	\end{align*}
	Assuming without loss of generality that $ 0 \leq x - y_{ j' } $, we can bound the nonnegatively indexed summands by \eqref{eq:GaussianBound} as
	\begin{align}
		\label{eq:NonnegativeSummands}
		g\left(x - y_{ j' } + \frac{k}{N}\right)
			&\leq \left(\frac{3}{c_1} + \frac{2}{\sqrt{ 2 \pi }}\right) \mathbbm{ e }^{ - \frac{(2 \pi)^2k^2}{2 c_1^2N^2} }.
	\end{align}
	For the negatively indexed summands, the definition of $ j' = \arg \min_j |x - y_j| $ implies that $ x - y_{ j' } \leq \frac{1}{2N} $.
	In particular,
	\begin{equation*}
		x - y_{ j' } - \frac{k}{N} \leq \frac{1 - 2k}{2N} < 0 \implies \left(x - y_{ j' } - \frac{k}{N} \right)^2 \geq \frac{1 - 2k}{2N} \left( x- y_{ j' } - \frac{k}{N}  \right) \geq \frac{2k - 1}{2N} \cdot \frac{k}{N}  ,
	\end{equation*}
	giving
	\begin{align}
		\label{eq:NegativeSummands}
		g \left( x - y_{ j' } - \frac{k}{N} \right) \leq \left( \frac{3}{c_1} + \frac{2}{\sqrt{ 2\pi }} \right) \mathbbm{ e }^{ - \frac{(2\pi)^2k^2}{2c_1^2N^2} } \mathbbm{ e }^{ \frac{(2\pi)^2k}{4c_1^2N^2} }.
	\end{align}

	We now bound the final exponential. 
	We first recall from \cite{Mehri2019} the choices of parameters 
	\begin{gather*}
		c_1 = \frac{\beta \sqrt{ \ln N }}{N}, \quad  \kappa = \left \lceil \gamma \ln N \right \rceil + 1, \quad \gamma = \frac{6r}{\sqrt{ 2 }\pi} = \frac{\beta \sqrt{ r}}{2 \sqrt{ \pi }}, \quad \beta = 6 \sqrt{ r }, \quad\textrm{where } 1 \leq r \leq \frac{N}{36}.
	\end{gather*}
	For $ k \in [1, \kappa] \cap \mathbbm{ Z } $ then,
	\begin{align*}
		\exp \left( \frac{(2\pi)^2k}{4c_1^2N^2} \right) 
			&\leq \exp \left( \frac{(2\pi)^2\kappa}{4c_1^2N^2} \right) \\
			&\leq \exp \left( \frac{\pi^2\left( \frac{6r \ln N}{\sqrt{ 2 }\pi} + 2\right)}{36 r \ln N} \right) \\
			&\leq \exp \left( \frac{\pi}{6\sqrt{ 2 }} + \frac{\pi^2}{18r \ln N} \right) \\
			&\leq \exp \left( \frac{\pi}{6\sqrt{ 2 }} + \frac{\pi^2}{18 \ln 36} \right) =: A.
	\end{align*}

	Combining this with our bounds for the nonnegatively indexed summands \eqref{eq:NonnegativeSummands} and the negatively indexed summands \eqref{eq:NegativeSummands}, we have
	\begin{equation*}
	\frac{1}{N} \sum_{ k = -\kappa }^\kappa g(x - y_{ j'+ k }) \leq \left(\frac{3}{\beta \sqrt{ \ln N }} + \frac{1}{N \sqrt{ 2 \pi }}\right) \left( 1 + (1 + A) \sum_{ k = 1 }^\kappa \mathbbm{ e }^{ - \frac{(2\pi)^2 k^2}{2 \beta^2 \ln N}  } \right)
	\end{equation*}
	Expressing the final sum as a truncated lower Riemann sum and applying a change of variables on the resulting integral, we have
	\begin{equation*}
		\sum_{ k = 1 }^\kappa \mathbbm{ e }^{ - \frac{(2\pi)^2 k^2}{2 \beta^2 \ln N} } \leq \frac{\beta \sqrt{ \ln N }}{\sqrt{ 2 }\pi} \int_{ 0 }^{ \infty } e^{ -x^2 } \,d x = \frac{\beta \sqrt{ \ln N }}{2 \sqrt{ 2 \pi }}.
	\end{equation*}
	Making use of our parameter values from \cite{Mehri2019}, and the fact that $ 1 \leq r \leq \frac{N}{36} $,
	\begin{equation}
		\label{eq:ConstantBoundOnGaussianSum}
		\begin{aligned}
			\frac{1}{N} \sum_{ k = -\kappa }^\kappa g(x - y_{ j'+ k })
				&\leq \left(\frac{3}{\beta \sqrt{ \ln N }} + \frac{1}{N \sqrt{ 2 \pi }}\right) \left( 1 + \frac{1 + A}{2 \sqrt{ 2 \pi }} \beta \sqrt{ \ln N } \right) \\
				&\leq \frac{3}{6 \sqrt{ \ln 36 }} + \frac{3(1 + A)}{2 \sqrt{ 2 \pi }} + \frac{1}{36 \sqrt{ 2\pi }} + \frac{1+A}{4\pi} \sqrt{ \frac{\ln 36}{36} } \\
				&< 2.
		\end{aligned}
	\end{equation}

	With our revised bound for \eqref{eq:NoisyConvolutionBound} above, we reprove \cite[Theorem 4]{Mehri2019} to estimate $ g \ast f $ by the truncated discrete convolution with noisy samples.
	In particular, we apply \cite[Theorem 3]{Mehri2019}, \eqref{eq:NoisyConvolutionBound}, \eqref{eq:GaussianBound}, and finally our same assumption that $ 1 \leq r \leq \frac{N}{36}  $ to obtain
	\begin{align*}
		&\left| (g * f) (x) - \frac{1}{N} \sum_{ j = j' - \left \lceil \frac{6r}{\sqrt{ 2 } \pi} \ln N  \right \rceil - 1 }^{ j' + \left \lceil \frac{6r}{\sqrt{ 2 }\pi} \ln N \right \rceil + 1 } (f(y_j) + \mu_j) g(x - y_j) \right|\\
			&\qquad\leq \frac{N^{ 1 - r }}{6 \sqrt{ r } \sqrt{ \ln N }} \| \mathbf{f} \|_\infty N^{ -r } + \left( \frac{3}{\sqrt{ 2\pi }} + \frac{1}{2\pi} \sqrt{ \frac{\ln 36}{36} } \right) \| \mathbf{f} \|_\infty N^{ -r } + 2 \| \boldsymbol{\mu} \|_\infty \\
			&\qquad\leq \left( \frac{1}{6 \sqrt{ \ln 36 }} + \frac{3}{\sqrt{ 2\pi }} + \frac{1}{2\pi } \sqrt{ \frac{\ln 36}{36} } \right)\frac{ \| \mathbf{f} \|_\infty }{N^r} + 2 \| \boldsymbol{\mu} \|_\infty < 2 \left( \frac{\| \mathbf{f} \|_\infty}{N^r} + \| \boldsymbol{\mu}\|_\infty \right).
	\end{align*}
	Replacing all references of $ 3 \| \f \|_{ \infty } N^{ -r }  $ by $ 2( \| \f \|_\infty N^{ -r } + \| \boldsymbol{\mu} \|_\infty ) $ in the remainder of the steps up to proving \cite[Theorem 5]{Mehri2019} gives the desired noise robustness (with a slightly improved constant).

	Using the revised error estimates of the nonequispaced algorithm from Theorem~\ref{thm:1dSFT} and redefining $ \delta = 3 (\| \hat \f - \hat \f_s^\mathrm{opt} \|_1 / 2s + 2( \| \f \|_\infty N^{ -r } +  \| \boldsymbol{ \mu } \|_\infty)) $ as in the proof of \cite[Theorem 5]{Mehri2019} (which also contains the proof of property (2)), the discretization algorithm \cite[Algorithm 1]{Mehri2019} will produce candidate Fourier coefficient approximations in lines 9 and 12 corresponding to every $ |\hat f_\omega| \geq (4 + 2 \sqrt{ 2 }) \delta $ in place of $ 4\delta $ in Theorem~\ref{thm:1dSFT}.
	The exact same argument as in the proof of Theorem~\ref{thm:1dSFT} then applies to the selection of the $ 2s $-largest entries of this approximation with the revised threshold values and error bounds to give properties \eqref{it:1dDiscreteSFT:FindSignificantFrequencies} and \eqref{it:1dDiscreteSFT:2Estimate}.

	In detail, \cite[Lemma 13]{Mehri2019} and the discussion right after its statement gives that property (2) holds for any approximate coefficient with frequency recovered throughout the algorithm (which, for the purposes of the following discussion, we will store in $ \mathbf{x} $ rather than $ \hat R $ defined in \cite[Algorithm 1]{Mehri2019}), not just those in the final output $ \mathbf{v} := \mathbf{x}_s^{ \mathrm{opt} } $.
	Additionally, by the same lemma and our revised bounds from Theorem~\ref{thm:1dSFT}, any frequency $ \omega \in [N] $ satisfying $ |f_\omega| > (4 + 2 \sqrt{ 2 })\delta $ will have an associated coefficient estimate in $ \mathbf{x} $.

	By Lemma~\ref{lem:ErrorInTauSigFrequencies}, $ | \mathcal{S}_{ (4 + 2 \sqrt{ 2 })\delta } | \leq 2s = | \supp(\mathbf{v}) | $, and so if $ \omega \in \mathcal{S}_{ (4 + 2 \sqrt{ 2 })\delta} \setminus \supp( \mathbf{v} ) $, there exists some $ \omega' \in \supp(\mathbf{v}) \setminus \mathcal{S}_{ (4 + 2 \sqrt{ 2 })\delta } $ such that $ v_{\omega'} $ took the place of $ v_{ \omega } $ in $ \mathcal{S} $.
	In particular, this means that $ |x_{ \omega' }| \geq |x_{ \omega }| $, $ |\hat f_{ \omega' }| \leq (4 + 2 \sqrt{ 2 })\delta $, and $ |\hat f_{ \omega }| > (4 + 2 \sqrt{ 2 }\delta) $.
	Thus,
	\begin{equation*}
		(4 + 2 \sqrt{ 2 })\delta + \sqrt{ 2 }\delta > |\hat f_{ \omega' }| + \sqrt{ 2 } \delta \geq |x_{ \omega' }| \geq |x_\omega| \geq |\hat f_{ \omega }| - \sqrt{ 2 }\delta,
	\end{equation*}
	implying that $ |\hat f_\omega| \leq 4(1 + \sqrt{ 2 })\delta $ and therefore proving (1).

	Finally, to prove (3), we use Lemma~\ref{lem:ErrorInTauSigFrequencies}, and consider
	\begin{align*}
		\| \hat \f - \mathbf{v} \|_2
			&\leq \| \hat \f - \hat \f \rvert_{ \supp (\mathbf{v}) } \|_2 - \| ( \hat \f - \mathbf{v} )\rvert_{ \supp (v) } \|_2 \\
			&\leq \| \hat \f - \hat \f\rvert_{ \mathcal{S}_{ (4 + 2 \sqrt{ 2 })\delta } \cap \supp (\mathbf{v}) } \|_2+ \sqrt{ 2 }\delta\sqrt{ 2s } \\
			&\leq \| \hat \f - \hat \f\rvert_{ \mathcal{S}_{ (4 + 2 \sqrt{ 2 })\delta } } \|_2 + \| \hat \f\rvert_{ \mathcal{S}_{ (4 + 2 \sqrt{ 2 })\delta } \setminus \supp (\mathbf{v}) }\|_2  + 2 \delta \sqrt{ s } \\
			&\leq \| \hat \f - \hat \f_{ 2s }^\mathrm{opt}\|_2 + (4 + 2 \sqrt{ 2 })\delta \sqrt{ 2s } + 4(1 + \sqrt{ 2 })\delta \sqrt{ 2s } + 2 \delta \sqrt{ s } \\
			&\leq \|\hat \f - \hat \f_{ 2s }^\mathrm{opt} \|_2 + (14 + 8 \sqrt{ 2 }) \delta \sqrt{ s }
	\end{align*}
	which finishes the proof.
 \end{proof}

\section{Fast multivariate sparse Fourier transforms}
\label{sec:multivariate-results}

Having detailed two sublinear-time, one-dimensional SFT algorithms, we are now prepared to extend these to the multivariate setting.
The general approach will be to apply the one-dimensional methods to transformations of our multivariate function of interest with samples taken along rank-1 lattices.
The particular approaches for transforming our multivariate function will then allow for the efficient extraction of multidimensional frequency information for the most energetic coefficients identified by univarate SFTs.
In particular, our first approach considered in Section~\ref{sub:PhaseEnc} successively shifts the function in each dimension, whereas our second approach considered in Section~\ref{sub:2dDFT} successively collapses all but one dimension along a rank-1 lattice and samples the resulting two-dimensional function.

Before continuing, it is important to stress that these two approaches given in Algorithms~\ref{alg:PhaseEnc} and \ref{alg:2dDFT} below can make use of \emph{any} univariate SFT algorithm $ \mathcal{A}_{ s,M } $.
Thus, the analysis of each algorithm is presented in a similarly modular fashion.
Each algorithm is followed by a lemma (Lemma~\ref{lem:PhaseEncRecovery} and Lemma~\ref{lem:2dDFTRecovery} respectively) which provides associated error guarantees when any sufficiently accurate univariate SFT $ \mathcal{A}_{ s, M } $ is employed.
The lemmas are then each followed by two corollaries (Corollaries~\ref{cor:PhaseEncRecoveryDiscreteSFT} and \ref{cor:PhaseEncRecoverySublinearSFT} and Corollaries~\ref{cor:2dDFTRecoveryDiscreteSFT} and \ref{cor:2dDFTRecoverySublinearSFT} respectively) where we apply the lemma to the two example univariate SFTs reviewed in Section~\ref{sec:1dSFTs} specified by Theorems~\ref{thm:1dDiscreteSFT} and \ref{thm:1dSFT}.

\subsection{Phase encoding}
\label{sub:PhaseEnc}

We begin with a review of how the shift operator $ S_{ \ell, \alpha } $ interacts with the Fourier transform.

\begin{lem}
	\label{lem:ShiftedFunctionCoefficients}
	For any dimension $ \ell \in [d] $, shift $ \alpha \in \mathbbm{ R } $, and function $ f : \mathbbm{ T }^d \rightarrow \C $, shifting $ f $ with the operator $ S_{ \ell, \alpha } $ modulates the Fourier coefficients as $ c_\k(S_{ \ell, \alpha } f) = \mathbbm{ e }^{ 2 \pi \mathbbm{ i } k_\ell \alpha } \, c_\k(f)$ for all $ \k \in \mathbbm{ Z }^d $.
\end{lem}
\begin{proof}
	Using the Fourier series of $ f $ in the computation of $ c_\k(S_{ \ell, \alpha } f) $ gives
	\begin{align*}
		c_\k(S_{ \ell, \alpha } f) 
			&= \int_{ \mathbbm{ T }^d } S_{ \ell, \alpha } f(\x) \, \mathbbm{ e }^{ -2 \pi \mathbbm{ i } \k \cdot \x  } \; \mathrm{d} \x \\
			&= \int_{ \mathbbm{ T }^d } \sum_{ \h \in \mathbbm{ Z }^d } c_\h(f) \, \mathbbm{ e }^{ 2 \pi \mathbbm{ i } (h_\ell (x_\ell + \alpha) + \h_\ell' \cdot \x_\ell') } \, \mathbbm{ e }^{ -2 \pi \mathbbm{ i } \k \cdot \x } \;\mathrm{d} \x \\
			&= \sum_{ \h \in \mathbbm{ Z }^d } \mathbbm{ e }^{ 2 \pi \mathbbm{ i } h_\ell\alpha } \, c_\h(f) \int_{ \mathbbm{ T }^{ d } } \mathbbm{ e }^{ 2 \pi \mathbbm{ i } (\h - \k) \cdot \x } \; \mathrm{d} \x \\
			&= \mathbbm{ e }^{ 2 \pi \mathbbm{ i } k_\ell \alpha } \, c_\k(f),
	\end{align*}
	as desired.
\end{proof}

By performing additional lattice SFTs on shifted versions of our original function, we can then separate out the components of recovered frequencies in modulations of the function's Fourier coefficients.
Using the common one-dimensional recovered frequencies between each transform in the form $ \k \cdot \z \bmod M $, we can then reconstruct the multivariate frequencies.
This approach of encoding frequencies in the phase is summarized in Algorithm~\ref{alg:PhaseEnc}.

\begin{algorithm}[H]
	\caption{Simple Frequency Index Recovery by Phase Encoding}
	\label{alg:PhaseEnc}
	\begin{algorithmic}[1]
		\REQUIRE A multivariate periodic function $ f \in \mathcal{W}(\mathbbm{ T }^d \cap C (\mathbbm{ T }^d) $ (from which we are able to obtain potentially noisy samples), a multivariate frequency set $ \mathcal{I} \subset \mathcal{B}_N^d $, a reconstructing rank-1 lattice $ \Lambda(\z, M) $ for $ \mathcal{I} $, and an SFT algorithm $ \mathcal{A}_{ s, M } $.
		\ENSURE Sparse coefficient vector $ \mathbf{b} = (b_\k)_{ \k \in \mathcal{B}_N^d } $ (optionally supported on $ \mathcal{I} $, see Line~\ref{alg:PhaseEnc:Check}), an approximation to $ \left( c\rvert_{ \mathcal{I} } \right)_s^\mathrm{ opt } $.
		\STATE Apply $ \mathcal{A}_{ s, M } $ to the univariate restriction of $ f $ to the lattice, $ a(t) = f(t \z) $, to produce $ \v = \mathcal{A}_{ s, M } a $, a sparse approximation of $ \F_M \, \a \in \C^M $. \label{alg:PhaseEnc:LatticeSFT}
		\FORALL{$\ell \in [d]$}
		\STATE Apply $ \mathcal{A}_{ s, M } $ to $ a^\ell(t) = S_{ \ell, 1/N }f(t\z) $ to produce $ \v^\ell = \mathcal{A}_{ s, M } a^\ell $, a sparse approximation of $ \F_M \, \a^\ell \in \C^M $.
			\label{alg:PhaseEnc:ShiftedLatticeSFTs}
		\ENDFOR
		\STATE $ \mathbf{b} \gets \mathbf{0} $
		\FORALL{$\omega \in \supp (\v) \subset \mathcal{B}_M$} \label{alg:PhaseEnc:BeginRecovery}
			\FORALL{$\ell \in [d]$}
			\STATE $ (k_{ \omega})_{ \ell } \gets \mathrm{round}(N \arg( v^\ell_\omega / v_\omega) / 2\pi) $ 
			\ENDFOR
			\IF{$\k_\omega \cdot \z \equiv \omega \imod{M}$ (and optionally $ \k_\omega \in \mathcal{I}$; see Remark~\ref{rem:AssignOnlyToI})} \label{alg:PhaseEnc:Check}
				\STATE $ b_{ \k_\omega } \gets b_{ \k_\omega } + v_\omega $ 
			\ENDIF
			\ENDFOR \label{alg:PhaseEnc:EndRecovery}
	\end{algorithmic}
\end{algorithm}

\subsubsection{Analysis of Algorithm~\ref{alg:PhaseEnc}}

\begin{lem}[General recovery result for Algorithm~\ref{alg:PhaseEnc}]
	\label{lem:PhaseEncRecovery}
	Let $ \mathcal{A}_{ s, M } $ in the input to Algorithm~\ref{alg:PhaseEnc} be a noise-robust SFT algorithm which, for a function $ a \in \mathcal{W}(\mathbbm{ T }) \cap C( \mathbbm{ T } ) $ corrupted by some arbitrary noise $ \mu: \mathbbm{T} \rightarrow \C $, constructs an $ s $-sparse Fourier approximation $ \mathcal{A}_{ s, M } (a + \mu) =: \v \in \C^M $ which
	\begin{enumerate}
		\item reconstructs every frequency (up to $ s $ many) of $ \F_M \, \a \in \C^M $, $ \omega \in \mathcal{B}_M$, with corresponding Fourier coefficient meeting the tolerance $ |(\F_M \, \a)_\omega| > \tau $, 
		\item satisfies the $ \ell^\infty $ error estimate for recovered coefficients
		\begin{equation*}
			\left\| (\F_M \, \a - \v)\rvert_{ \supp (\v) } \right\|_\infty \leq \eta_\infty < \tau,
		\end{equation*}
		\item satisfies the $ \ell^2 $ error estimate
		\begin{equation*}
			\left\| \F_M \, \a - \v \right\|_2 \leq \eta_2,
		\end{equation*}
		\item and requires $ \mathcal{O}(P(s, M)) $ total evaluations of $ a $, operating with computational complexity $\mathcal{O}(R(s, M))$.
	\end{enumerate}
	Additionally, assume that the parameters $ \tau $ and $ \eta_\infty $ hold uniformly for each SFT performed in Algorithm~\ref{alg:PhaseEnc}.

	Let $ f $, $ \mathcal{I} $, and $ \Lambda(\z, M) $ be as specified in the input to Algorithm~\ref{alg:PhaseEnc}.
	Collecting the $ \tau $-significant frequencies of $ f $ into the set $ \mathcal{S}_\tau := \{ \k \in \mathcal{I} \mid |c_\k (f) | > \tau \} $, assume that $ |\supp( c ) \cap \mathcal{S}_\tau | \leq s $, and set 
	\begin{equation*}
		\beta = \max \left( \tau, \eta_\infty \left( 1 + \frac{2}{\sin \left( \frac{\pi}{N} \right)} \right) \right).
	\end{equation*}
	Then Algorithm~\ref{alg:PhaseEnc} (ignoring the optional check on Line~\ref{alg:PhaseEnc:Check}) will produce an $ s $-sparse approximation~$\mathbf{b}$ of the Fourier coefficients of~$ f $ satisfying the error estimate
	\begin{align*}
		\left\| \mathbf{b} - c(f) \right\|_{ \ell^2( \mathbbm{ Z }^d )}
		&\leq \eta_2 + (\beta + \eta_\infty) \sqrt{ \max(s - |\mathcal{S}_\beta|, 0) } \\
		&\qquad+ \| c(f)\rvert_{ \mathcal{I} } - c(f) \rvert_{ \mathcal{S}_\beta } \|_{ \ell^2( \mathbbm{ Z }^d ) } + \| c(f) - c(f)\rvert_{ \mathcal{ I } } \|_{ \ell^2(\mathbbm{ Z }^d) }
	\end{align*}
	requiring $ \mathcal{O} \left( d \cdot P(s, M) \right) $ total evaluations of $ f $, in $ \mathcal{O}\left(d \cdot (R(s, M) + s) \right) $ total operations.
\end{lem}

\begin{proof}
	We begin by assuming that $ f $ is a trigonometric polynomial with $ \supp (c(f)) \subset \mathcal{I} $.
	Since $ \Lambda(\z, M) $ is a reconstructing rank-1 lattice for $ \mathcal{I} $, Lemmas~\ref{lem:DFTEqualsFT} and \ref{lem:DFTAliasing} ensure that for each $ \k \in \mathcal{I} $, $ c_\k(f) = \hat a_{ \k \cdot \z \bmod M } = (\F_M \, \a)_{ \k \cdot \z \bmod M } $.
	Thus, Lines~\ref{alg:PhaseEnc:LatticeSFT} and \ref{alg:PhaseEnc:ShiftedLatticeSFTs} of Algorithm~\ref{alg:PhaseEnc} will produce coefficient estimates in the lattice DFT for every $ \k \in \mathcal{S}_\tau $.
	We then write these SFT approximations as $  v_{ \k \cdot \z \bmod M } = c_\k(f) + \eta_\k $ and $ v_{ \k \cdot \z \bmod M }^\ell = \mathbbm{ e }^{ 2\pi \mathbbm{ i }k_\ell / N }(c_\k(f) + \eta_\k^\ell) $ respectively, where we have made use of Lemma~\ref{lem:ShiftedFunctionCoefficients}.
	Note that $ |\eta_\k|, |\eta_\k^\ell| \leq \eta_\infty $.
	Now, considering the estimate for $ k_\ell $, we have
	\begin{align*}
		\frac{N}{2 \pi} \arg \left( \frac{v_{ \k \cdot \z \bmod M }^\ell}{v_{ \k \cdot \z \bmod M }} \right) 
		&= \frac{N}{2 \pi} \arg \left( \mathbbm{ e }^{ 2 \pi \mathbbm{ i } k_\ell / N } \frac{c_\k(f) + \eta_\k^\ell}{v_{ \k \cdot \z \bmod M }}  \right) \\
		&= k_\ell +  \frac{N}{2 \pi} \arg \left( \frac{c_\k(f) + \eta_\k^\ell}{v_{ \k \cdot \z \bmod M }}  \right) \\
		&= k_\ell + \frac{N}{2 \pi} \arg \left( 1 + \frac{\eta_\k^\ell - \eta_\k}{v_{ \k \cdot \z \bmod M }} \right).
	\end{align*}
	We now only consider $ |c_\k| > \beta \geq \max(\tau, 3\eta_\infty) $, that is $ \k \in \mathcal{S}_{ \beta } \subset \mathcal{S}_\tau $, and therefore, the corresponding approximate coefficient satisfies $ |v_{ \k \cdot \z \bmod M } | > \beta - \eta_\infty $.
	Thus, the magnitude of the fraction in the argument must be strictly less than $ \frac{2\eta_\infty}{\beta - \eta_\infty} \leq 1$.
	Therefore, we consider the argument of a point lying in the right half of the complex plane, in the open disc of radius $ \frac{2\eta_\infty}{\beta - \eta_\infty} $ centered at $ 1 $.
	The maximal absolute argument of a point in this disc will be that of a point lying on a tangent line passing through the origin.
	This point, the origin, and $ 1 $ then form a right triangle from which we deduce that 
	\begin{equation*}
		\left| \arg \left( 1 + \frac{\eta_\k^\ell - \eta_\k}{v_{ \k \cdot z \bmod M }}  \right) \right| < \arcsin \left( \frac{2 \eta_\infty}{\beta - \eta_\infty}  \right) \leq \frac{\pi}{N}.
	\end{equation*}
	Our choice of $ \beta \geq \eta_\infty (1 + 2 / \sin(\pi/N)) $ then implies that
	\begin{equation*}
		\left| \frac{N}{2\pi} \arg \left( \frac{v_{ \k \cdot \z \bmod M }^\ell}{v_{ \k \cdot \z \bmod M }}  \right) - k_\ell \right| < \frac{1}{2},
	\end{equation*}
	and so after rounding to the nearest integer, Algorithm~\ref{alg:PhaseEnc} will recover $ k_\ell $ for all $ \ell \in [d] $ and $ \k \in \mathcal{S}_{ \beta } $.

	By this mapping constructed in the final loop of Algorithm~\ref{alg:PhaseEnc}, we set $ b_{ \k } := v_{ \k \cdot \z \bmod M } $  for each $ \k \in \mathcal{S}_{ \beta } $. 
	Additionally, the $ \max(s - | \mathcal{S}_\beta|, 0) $ many coefficients $ v_{ \omega } $ for which $ \omega \neq \k \cdot \z \bmod M $ for any $ \k \in \mathcal{S}_\beta $ are still available for potential assignment.
	If any multivariate frequency $ \k_\omega \in \mathcal{I} $ is reconstructed and passes the mandatory check in Line~\ref{alg:PhaseEnc:Check} then the approximate Fourier coefficient $ v_\omega $ properly corresponds to $ (\F_M \, \a)_{ \k_\omega \cdot \z \bmod M } = c_{ \k_\omega }(f) $.

	On the other hand, if some error introduced in the SFTs reconstructs a multivariate frequency $ \k_\omega \notin \mathcal{I} $, the reconstructing property does not allow us to conclude anything about a $ (k_\omega , \omega) $ pair passing the check in Line~\ref{alg:PhaseEnc:Check}.
	Thus, it is possible that $ v_\omega $ will contribute to some component of $ \mathbf{b} $ not corresponding to any frequency in $ \mathcal{I} $.
	At the least however, since we know that all entries of $ \v $ corresponding to frequencies in $ \mathcal{S}_\beta $ are correctly assigned, the remaining ones satisfy $ |v_\omega| \leq \beta + \eta_\infty $.
	Using these facts allows us to estimate the error as 
	\begin{equation}
		\label{eq:PolyPhaseEncEstimate}
		\begin{aligned}
			\| \mathbf{b} - c \|_{ \ell^2( \mathbbm{ Z }^d) } 
				&\leq \| \mathbf{b}\rvert_{ \mathbbm{ Z }^d \setminus \mathcal{I} } \|_{ \ell^2( \mathbbm{ Z }^d ) } + \| \mathbf{b}\rvert_{ \mathcal{I} } - c\rvert_{ \supp (\mathbf{b}) \cap \mathcal{I} } \|_{ \ell^2 (\mathbbm{ Z }^d ) } + \| c - c\rvert_{ \mathcal{S}_\beta } \|_{ \ell^2 ( \mathbbm{ Z }^d ) } \\
				&\leq (\beta + \eta_\infty) \sqrt{ \max(s - | \mathcal{S}_\beta|, 0) } + \eta_2 + \| c - c \rvert_{ \mathcal{S}_\beta } \|_{ \ell^2( \mathcal{I}) }
		\end{aligned}
	\end{equation}
	where we have additionally used the accuracy of the initial one-dimensional SFT and the assumption that $ c $ is supported on $ \mathcal{I} $.

	We now handle the case when $ f $ is not necessarily a polynomial with Fourier support contained in $ \mathcal{I} $.
	Rather than aiming to approximate $ c_\k(f) $ for every $ \k \in \mathbbm{ Z }^d $, we restrict attention to only frequencies in $ \mathcal{I} $, instead attempting to approximate the Fourier coefficients of $ f_{ \mathcal{I} } = \sum_{ \k \in \mathcal{I} } c_\k(f) \mathbbm{ e }^{ 2\pi \mathbbm{ i } \k \cdot \circ } $.
	We then have that $ f =: f_{ \mathcal{I} } + f_{ \mathbbm{ Z }^d \setminus \mathcal{I} } $ and view potentially noisy input $ f + \mu $ to our algorithm as 
	\begin{equation*}
		f + \mu = f_{ \mathcal{I} } + \underbrace{f_{ \mathbbm{ Z }^d \setminus \mathcal{I} } + \mu}_{\mu'}.
	\end{equation*}

	Algorithm~\ref{alg:PhaseEnc} applied to $ f + \mu $ is then equivalent to applying it to $ f_{ \mathcal{I} } + \mu'$, where now $ \tau $, $ \eta_\infty $, and $ \eta_2 $ depend on $ \mu' $, and the output is an approximation of $ c \rvert_{ \mathcal{I} } $.
	Since $ \mu' $ represents noise on the input to $ \mathcal{A}_{ s, M } $ in its applications to $ f_{ \mathcal{I} }(t\z) $ and $ S_{ \ell, 1/N }f_{ \mathcal{I} }(t \z) $ we remark here that
	\begin{equation}
		\label{eq:TruncationAsMeasurementNoise}
		\| \mu' \|_\infty \leq \|f_{ \mathbbm{ Z }^d \setminus \mathcal{I} } \|_\infty + \| \mu \|_\infty \leq \| c(f) - c(f) \rvert_{ \mathcal{I} } \|_{ \ell^1( \mathbbm{ Z }^d ) } + \| \mu \|_\infty
	\end{equation}
	so as to help us estimate $ \tau $, $ \eta_\infty $, and $ \eta_2 $ in future applications of the lemma.
	Accounting for the truncation to $ \mathcal{I} $ in the $ \ell^2 $ error bound and using \eqref{eq:PolyPhaseEncEstimate} applied to $ c \rvert_{ \mathcal{I} } $, we estimate
	\begin{align*}
		\| \mathbf{b} - c \|_{ \ell^2( \mathbbm{ Z }^d) } 
			&\leq \| \mathbf{b} - c\rvert_{ \mathcal{I} } \|_{ \ell^2( \mathbbm{ Z }^d ) } + \| c - c\rvert_{ \mathcal{I} }\|_{ \ell^2( \mathbbm{ Z }^d) } \\
			&\leq (\beta + \eta_\infty) \sqrt{ \max(s - | \mathcal{S}_\beta|, 0) } + \eta_2 + \| c\rvert_{ \mathcal{I} } - c \rvert_{ \mathcal{S}_\beta } \|_{ \ell^2( \mathbbm{ Z }^d ) } + \| c - c\rvert_{ \mathcal{I} }\|_{ \ell^2( \mathbbm{ Z }^d) }.
	\end{align*}

	Since $ 1 + d $ SFTs are required, the number of $ f $ evaluations is $ \mathcal{O} \left( d \cdot P(s, M) \right) $ and the associated computational complexity is $ \mathcal{O} \left( d \cdot R(s, M) \right) $.
	The complexity of Lines~\ref{alg:PhaseEnc:BeginRecovery}--\ref{alg:PhaseEnc:EndRecovery} is $ \mathcal{O}(sd) $.
\end{proof}

\begin{rem}
	\label{rem:AssignOnlyToI}
	Since the only possible misassigned values of $ v_\omega $ contribute to coefficients in $ \mathbf{b} $ outside the chosen frequency set $ \mathcal{I} $ for which $ \Lambda(\z, M) $ is reconstructing, if it is possible to quickly (e.g., in $ \mathcal{O}(d) $ time) check a multivariate frequency's inclusion in $ \mathcal{I} $ (e.g., a hyperbolic cross), entries outside of $ \mathcal{I} $ in $ \mathbf{b} $ can be identified in the optional check on Line~\ref{alg:PhaseEnc:Check} and remain (correctly) unassigned.
	This has the effect of removing the $ (\beta + \eta_\infty) \sqrt{ \max(s - | \mathcal{S}_\beta |, 0) } $ term in the error bound while not increasing the computational complexity.
	Additionally, this outputs an approximation to $ (c\rvert_{ \mathcal{I} })_s^\mathrm{opt} $ which is supported only on our supplied frequency set $ \mathcal{I} $ as we may expect or prefer.
\end{rem}

We now apply Lemma~\ref{lem:PhaseEncRecovery} with the discrete sublinear-time SFT from Theorem~\ref{thm:1dDiscreteSFT} to give specific error bounds in terms of best $ s $-term approximation errors as well as detailed runtime and sampling complexities.

\begin{cor}[Algorithm~\ref{alg:PhaseEnc} with discrete sublinear-time SFT]
	\label{cor:PhaseEncRecoveryDiscreteSFT}
	Let $ N \geq 9 $.
	For $ \mathcal{I} \subset \mathcal{B}_N^d $ with reconstructing rank-1 lattice $ \Lambda(\z, M) $ and the function $ f \in \mathcal{W}(\mathbbm{ T }^d) \cap C( \mathbbm{ T }^d ) $, we consider applying Algorithm~\ref{alg:PhaseEnc} where each function sample may be corrupted by noise at most $ e_\infty \geq 0 $ in absolute magnitude. Using the discrete sublinear-time SFT algorithm $ \mathcal{A}_{ 2s, M }^\mathrm{disc} $ or $ \mathcal{A}_{ 2s, M }^\mathrm{disc, MC} $ with parameter $ 1 \leq r \leq \frac{M}{36} $, Algorithm~\ref{alg:PhaseEnc} will produce $ \mathbf{b} = (b_\k)_{ \k \in \mathcal{B}_N^d } $ a $ 2s $-sparse approximation of $ c $ satisfying the error estimate
	\begin{align*}
		\| \mathbf{b} &- c \|_2 & \\
				  &\leq \|c - c\rvert_{ \mathcal{I} }\|_2 + \left( 48 + 4N \right) \frac{\|c \rvert_{ \mathcal{I} } - \left( c \rvert_{ \mathcal{I} } \right) _s^\mathrm{opt} \|_1}{\sqrt{ s }} + \left( 188 + 16 N \right) \sqrt{ s } (\|f\|_\infty M^{ -r } + \|c - c\rvert_{ \mathcal{I} }\|_1 + e_\infty ) \\
	\end{align*}
	albeit with probability $ 1 - \sigma \in [0, 1) $ for the Monte Carlo version.
	The total number of evaluations of $ f $ and computational complexity will be
	\begin{equation*}
		\mathcal{O} \left( \frac{ds^2 r^{ 3/2 } \log^{ 11/2 }M}{\log s}  \right) \text{ or } \mathcal{O} \left( ds r^{ 3/2 } \log^{ 9/2 } M \log \left( \frac{dM}{\sigma}  \right) \right)
	\end{equation*}
	for $ \mathcal{A}_{ 2s, M }^\mathrm{ disc } $ or $ \mathcal{A}_{ 2s, M }^\mathrm{ disc, MC } $ respectively.
\end{cor}
\begin{proof}
	For the definitions of $ \tau $ and $ \beta $ in Lemma~\ref{lem:PhaseEncRecovery} with associated values given by Theorem~\ref{thm:1dDiscreteSFT}, Lemma~\ref{lem:ErrorInTauSigFrequencies} applied with $ \x = c\rvert_{ \mathcal{I} } $ implies that $ \mathcal{S}_\beta $ can contain at most $ 2s $ elements, and we have the bound
	\begin{equation}
		\label{eq:LargestCoefficientsErrorCorollary}
		\| c\rvert_{ \mathcal{I} } - c\rvert_{ \mathcal{S}_\beta } \|_{ \ell^2( \mathbbm{ Z }^d) } \leq \| c\rvert_{ \mathcal{I} } - (c\rvert_{ \mathcal{I} })_{ 2s }^\mathrm{opt} \|_{ \ell^2( \mathbbm{ Z }^d )} + \beta \sqrt{ 2s } \leq \frac{\|c\rvert_{ \mathcal{I} } - (c\rvert_{ \mathcal{I} })_s^\mathrm{opt}\|_{ \ell^1( \mathbbm{ Z }^d) }}{2 \sqrt{ s }} + \beta \sqrt{ 2s },
	\end{equation}
	where the last inequality follows from \cite[Theorem 2.5]{Foucart2013} applied to $ c\rvert_{ \mathcal{I} } - (c\rvert_{ \mathcal{I} })_s^\mathrm{opt} $.
	Lemma~\ref{lem:PhaseEncRecovery} then holds with $ s $ replaced by $ 2s $ for the $ 2s $-sparse approximations given by $ \mathcal{A}_{ 2s, M }^\mathrm{disc} $ or $ \mathcal{A}_{ 2s, M }^\mathrm{disc} $ in Algorithm~\ref{alg:PhaseEnc}.

	Assuming $ N \geq 9 $, the specific values of $ \tau $ and $ \eta_\infty $ from Theorem~\ref{thm:1dDiscreteSFT} give
	\begin{equation*}
		\beta = \max \left( \tau, \eta_\infty \left( 1 + \frac{2}{\sin \left( \frac{\pi}{N} \right)}  \right) \right) = \eta_\infty \left( 1 + \frac{2}{\sin \left( \frac{\pi}{N} \right)}  \right) \leq \eta_\infty \left( 1 + \frac{2}{9 \sin \left( \frac{\pi}{9} \right)} N \right).
	\end{equation*}
	Using our bound \eqref{eq:TruncationAsMeasurementNoise} from treating the truncation error as measurement noise additionally accounting for any noise in our input bounded by $ e_\infty $ we obtain
	\begin{align*}
		\beta 
			&\leq \eta_\infty \left(1 + \frac{2}{9 \sin \left( \frac{\pi}{9}  \right)} N \right) \\
			&\leq 3 \sqrt{ 2 } \left( \frac{\| c\rvert_{ \mathcal{I} }- \left( c \rvert_{ \mathcal{I} } \right)_s^\mathrm{opt} \|_1 }{2s} + 2 (\| f \|_\infty M^{ -r } + \|c - c\rvert_{ \mathcal{I} }\|_1 + e_\infty) \right) \left( 1 + \frac{2}{9 \sin \left( \frac{\pi}{9}  \right)} N \right).
	\end{align*}
	Inserting the estimate for $ \|c\rvert_{ \mathcal{I} } - c\rvert_{ \mathcal{S}_\beta } \|_2 $ from \eqref{eq:LargestCoefficientsErrorCorollary}, our bound for $ \beta $ above, and the value for $ \eta_2 $ from Theorem~\ref{thm:1dDiscreteSFT} (where again we use \cite[Theorem 2.5]{Foucart2013}) into the recovery bound in Lemma~\ref{lem:PhaseEncRecovery} gives the final error estimate.

	In detail, let
	\begin{gather*}
		A = \frac{\|c\rvert_{ \mathcal{I} } - (c\rvert_{ \mathcal{I} })_s^\mathrm{opt}\|_1}{\sqrt{ s }}, \quad \delta = \left( \frac{A}{2\sqrt{ s }} + 2( \|f\|_\infty M^{ -r } + \|c - c\rvert_{ \mathcal{I} } \|_1 + e_\infty) \right).
	\end{gather*}
	Then
	\begin{gather*}
		\beta \leq 3 \sqrt{ 2 } \left( 1 + \frac{2}{9 \sin \left( \frac{\pi}{9} \right)} N \right) \delta,\quad
		\eta_2 \leq \frac{A}{2} + 76 \sqrt{ s } \delta,\quad
		\|c\rvert_{ \mathcal{I} } - c\rvert_{ \mathcal{S}_\beta } \|_2 \leq \frac{A}{2} + \sqrt{ 2 } \beta \sqrt{ s }.
	\end{gather*}
	Our error bound is then
	\begin{align*}
		\| \mathbf{b} - c \|_2
			&\leq \eta_2 + (\beta + \eta_\infty) \sqrt{ 2s } + \| c \rvert_{ \mathcal{I} } - c \rvert_{ \mathcal{S}_\beta } \|_2 + \| c - c\rvert_{ \mathcal{I} }\|_2 \\
			&\leq A + 76 \sqrt{ s } \delta + 6\left( 3 + \frac{4}{9 \sin \left( \frac{\pi}{9} \right)} N \right)\sqrt{ s }\delta + \| c - c\rvert_{ \mathcal{I} }\|_2 \\
			&= A + \left( 94 + \frac{8}{3 \sin \left( \frac{\pi}{9} \right) } N \right)\sqrt{ s } \delta + \| c - c\rvert_{ \mathcal{I} }\|_2 \\
			&= \left( 48 + \frac{4}{3 \sin \left( \frac{\pi}{9} \right)} N \right)A + \left( 188 + \frac{16}{3 \sin \left( \frac{\pi}{9} \right)} N \right) \sqrt{ s } ( \| f \|_\infty M^{ -r } + \|c - c\rvert_{ \mathcal{I} }\|_1 + e_\infty ) \\
			&\qquad + \| c - c\rvert_{ \mathcal{I} }\|_2\\
			&\leq \left( 48 + 4N \right) \frac{\|c \rvert_{ \mathcal{I} } - \left( c \rvert_{ \mathcal{I} } \right) _s^\mathrm{opt} \|_1}{\sqrt{ s }} + \left( 188 + 16 N \right) \sqrt{ s } (\|f\|_\infty M^{ -r } + \|c - c\rvert_{ \mathcal{I} }\|_1 + e_\infty ) \\
			&\qquad+ \|c - c\rvert_{ \mathcal{I} }\|_2.
	\end{align*}

	The change to the complexity of the random algorithm arises from distributing the probability of failure $ \sigma$ over the $ d + 1 $ SFTs in a union bound.
\end{proof}

Though the nonequispaced SFTs discussed in Theorem~\ref{thm:1dSFT} do not approximate the discrete Fourier transform and therefore do not alias the one-dimensional frequencies $ \k \cdot \z $ into frequencies in $ \mathcal{B}_M $, slightly modifying Algorithm~\ref{alg:PhaseEnc} to use SFTs with a larger bandwidth allows for the following recovery result.

\begin{cor}[Algorithm~\ref{alg:PhaseEnc} with nonequispaced sublinear-time SFT]
	\label{cor:PhaseEncRecoverySublinearSFT}
	For $ \mathcal{I} \subset \mathcal{B}_N^d $ with $ N \geq 6 $, fix the new bandwidth parameter $ \tilde M := 2 \max_{ \k \in \mathcal{I} } |\k \cdot \z| + 1 $.
	For $ \Lambda(\z, M) $, a reconstructing rank-1 lattice for $ \mathcal{I} $ with $ M \leq \tilde M $, and the function $ f \in \mathcal{W}(\mathbbm{ T }^d) \cap C( \mathbbm{ T }^d ) $, we consider applying Algorithm~\ref{alg:PhaseEnc} where each function sample may be corrupted by noise at most $ e_\infty \geq 0 $ in absolute magnitude with the following modifications:
	\begin{enumerate}
		\item use the sublinear-time SFT algorithm $ \mathcal{A}_{ 2s, \tilde M }^\mathrm{sub} $ or $ \mathcal{A}_{ 2s, \tilde M }^\mathrm{sub, MC} $ 
		\item and only check equality against $ \omega $ in Line~\ref{alg:PhaseEnc:Check} (rather than equivalence modulo $ M $),
	\end{enumerate}
	to produce $ \mathbf{b} = (b_\k)_{ \k \in \mathcal{B}_N^d } $ a $ 2s $-sparse approximation of $ c $ satisfying the error estimate
	\begin{align*}
		\| \mathbf{b} - c \|_{ \ell^2( \mathbbm{ Z }^d) }
			&\leq \left( 24 + 3N \right) \left[ \frac{\|c \rvert_{ \mathcal{I} } - \left( c \rvert_{ \mathcal{I} } \right) _s^\mathrm{opt} \|_1}{\sqrt{ s }} + \sqrt{ s } \|c - c\rvert_{ \mathcal{I} }\|_1 + \sqrt{ s } e_\infty \right] + \|c - c\rvert_{ \mathcal{I} }\|_2.
	\end{align*}
	albeit with probability $ 1 - \sigma \in [0, 1) $ for the Monte Carlo version.
	For $ \mathcal{A}_{ 2s, \tilde M }^\mathrm{sub} $ and $ \mathcal{A}_{ 2s, \tilde M }^\mathrm{sub, MC} $ respectively, the total number of evaluations of $ f $ and computational complexity will be
	\begin{equation*}
		\mathcal{O} \left( \frac{d s^2 \log^4 \tilde M}{\log s}  \right) \text{ or } \mathcal{O} \left( d s \log^3(\tilde M) \log \left( \frac{d \tilde M}{\sigma}  \right) \right).
	\end{equation*}
\end{cor}
\begin{proof}
	The bandwidth specified ensures that $ \mathcal{B}_{ \tilde M } \supset \{\k \cdot \z \mid \k \in \mathcal{I}\} $. 
	In the case where $ f $ is a trigonometric polynomial with $ \supp (c(f)) \subset \mathcal{I} $, so long as there exists some $ M $ such that $ \Lambda(\z, M) $ is reconstructing for $ \mathcal{I} $, the one-dimensional Fourier transforms truncated to $ \mathcal{B}_{ \tilde M } $ coincide with length $ \tilde M $ DFTs.
	Thus, we can view an approximation from the algorithm in Theorem~\ref{thm:1dSFT} as one of a length $ \tilde M $ DFT.
	The reasoning in the proofs of Lemma~\ref{lem:PhaseEncRecovery} and Corollary~\ref{cor:PhaseEncRecoveryDiscreteSFT} then holds with the SFT algorithms, parameters, numbers of samples, and complexities of Theorem~\ref{thm:1dSFT}.

	In detail, we first note that for $ N \geq 6 $, we have
	\begin{equation*}
		\beta = \eta_\infty \left( 1 + \frac{2}{\sin \left( \frac{\pi}{N}  \right)}  \right) \leq \eta_\infty \left( 1 + \frac{2}{ 6 \left( \sin\frac{\pi}{6} \right) } N \right) = \eta_\infty \left( 1 + \frac{2}{3} N \right).
	\end{equation*}
	Now let
	\begin{gather*}
		A = \frac{\|c\rvert_{ \mathcal{I} } - (c\rvert_{ \mathcal{I} })_s^\mathrm{opt}\|_1}{\sqrt{ s }}, \quad \delta = \left( \frac{A}{\sqrt{ s }} + \|c - c\rvert_{ \mathcal{I} } \|_1 + e_\infty \right).
	\end{gather*}
	Then
	\begin{gather*}
		\beta \leq \sqrt{ 2 } \left( 1 + \frac{2}{3} N \right) \delta,\quad
		\eta_2 \leq \frac{A}{2} + (8 \sqrt{ 2 } + 6) \sqrt{ s } \delta,\quad
		\|c\rvert_{ \mathcal{I} } - c\rvert_{ \mathcal{S}_\beta } \|_2 \leq \frac{A}{2} + \sqrt{ 2 } \beta \sqrt{ s }.
	\end{gather*}
	Our error bound is then
	\begin{align*}
		\| \mathbf{b} - c \|_2
			&\leq \eta_2 + (\beta + \eta_\infty) \sqrt{ 2s } + \| c \rvert_{ \mathcal{I} } - c \rvert_{ \mathcal{S}_\beta } \|_2 + \| c - c\rvert_{ \mathcal{I} }\|_2 \\
			&\leq A + (8 \sqrt{ 2 } + 6) \sqrt{ s } \delta + \left( 6 + \frac{8}{3} N \right)\sqrt{ s }\delta + \| c - c\rvert_{ \mathcal{I} }\|_2 \\
			&= A + \left( 8 \sqrt{ 2 } + 12 + \frac{8}{3} N \right)\sqrt{ s } \delta + \| c - c\rvert_{ \mathcal{I} }\|_2 \\
			&= \left( 8 \sqrt{ 2 } + 13 + \frac{8}{3} N \right)A + \left( 8 \sqrt{ 2 } + 12 + \frac{8}{3} N \right) \sqrt{ s } ( \|c - c\rvert_{ \mathcal{I} }\|_1 + e_\infty ) + \| c - c\rvert_{ \mathcal{I} }\|_2\\
			&\leq \left( 24 + 3N \right) \left( \frac{\|c \rvert_{ \mathcal{I} } - \left( c \rvert_{ \mathcal{I} } \right) _s^\mathrm{opt} \|_1}{\sqrt{ s }} + \sqrt{ s } (\|c - c\rvert_{ \mathcal{I} }\|_1 + e_\infty )\right) + \|c - c\rvert_{ \mathcal{I} }\|_2.
	\end{align*}
\end{proof}

\begin{rem}
	\label{rem:Larger1dBandwidth}
	As in \cite{Gross2020}, we can estimate $ \tilde M $ above with two different techniques:
	\begin{gather*}
		\tilde M = 1 + 2\max_{ \k \in \mathcal{I} } \left| \sum_{ \ell \in [d] } k_\ell z_\ell \right| \leq 1 + 2\sum_{ \ell \in [d] } |z_\ell| \max_{ \k \in \mathcal{I} } |k_\ell|  = \mathcal{O}(dNM),\\
		\tilde M = 1 + 2\max_{ \k \in \mathcal{I} } \left| \sum_{ \ell \in [d] } k_\ell z_\ell \right| \leq 1 + 2 \| \z \|_\infty \max_{ \k \in \mathcal{I} } \| \k \|_1 = \mathcal{O}\left(M \max_{ \k \in \mathcal{I} } \| \k \|_1\right).
	\end{gather*}
	The latter case is especially useful when $ \mathcal{I} $ is a subset of a known $ \ell^1 $ ball as it will provide a dimension independent upper bound on $ \tilde M $.
	Either of these upper bounds may then be used in practice to avoid having to estimate $ \tilde M $.

	That being said however, if one is willing to perform the one-time search through the frequency set $ \mathcal{I} $ to more accurately calculate $ \tilde M $, one can go even further to use the minimal bandwidth $ \tilde M' =  \max_{ \k \in \mathcal{I} } (\k \cdot \z) - \min_{ \k \in \mathcal{I} } (\k \cdot \z) + 1 $ so long as the function samples are properly modulated to shift the one-dimensional frequencies into $ \mathcal{B}_{ \tilde M' } $.
	For example, running $ \mathcal{A}_{ 2s, \tilde M' }^\mathrm{sub} $ or $ \mathcal{A}_{ 2s, \tilde M' }^\mathrm{sub, MC} $ on $ a(t) = \mathbbm{ e }^{ 2 \pi \mathbbm{ i } \phi t }f(t\z) $ and $ a^\ell(t) = \mathbbm{ e }^{ 2\pi \mathbbm{ i } \phi t }S_{ \ell, 1/N }f(t\z) $ with $ \phi = \left \lfloor \frac{\tilde M'}{2} \right \rfloor - \max_{ \k \in \mathcal{I} }(\k \cdot \z) $ is acceptable so long as this shift is accounted for in the frequency check on Line~\ref{alg:PhaseEnc:Check}.
	Of course, these improvements will only have the effect of reducing the logarithmic factors in the computational complexity.
\end{rem}

\subsection{Two-dimensional DFT technique}
\label{sub:2dDFT}

Below, we will consider a method for recovering frequencies which, rather than shifting one dimension of the multivariate periodic function $ f $ at a time, leaves one dimension of $ f $ out at a time.
We will fix one dimension $ \ell \in [d] $ of $ f $ at equispaced nodes over $ \mathbbm{T} $ and apply a lattice SFT to the other $ d - 1 $ components. 
Applying a standard FFT to the results will produce a two-dimensional DFT.
The indices corresponding to the standard FFT will represent frequency components in dimension $ \ell $ while the indices corresponding to the lattice SFT will be used to synchronize with known one-dimensional frequencies $ \k \cdot \z \bmod M $.
The approach is summarized in the following lemma and in Algorithm~\ref{alg:2dDFT}.

\begin{lem}
	\label{lem:2d4supportFinding}
	Fix some finite multivariate frequency set $ \mathcal{I} \subset \mathcal{B}_N^d $, let $ \Lambda(\z, M) $ be a reconstructing rank-1 lattice for $ \{\k - k_\ell \mathbf{e}_\ell \mid \k \in \mathcal{I}\} $, and assume that $ f $ has Fourier support $ \supp (c) \subset \mathcal{I} $.
	Fixing one dimension $ \ell \in [d] $, and writing the generating vector as $ \z = (z_\ell, \z_\ell') \in \mathbbm{ Z }^d $, define the polynomials 
	\begin{equation*}
		a_j^\ell(t) := f \left( \frac{j}{N}, t \z_\ell' \right) \text{ for all $ j \in [N] $,} 
	\end{equation*}
	that is, fix coordinate $ \ell $ at $ j / N $ and restrict the remaining coordinates to dimensions $ [d] \setminus \{\ell\} $ of the rank-1 lattice.
	Then for all one-dimensional frequencies $ \omega \in [M] $,
	\begin{equation*}
		\left(\F_M \, \a_j^\ell\right)_{ \omega } = 
		\begin{cases}
			\sum\limits_{ \substack{h_\ell \in \mathcal{B}_N \text{ s.t.} \\(h_\ell, \k_\ell') \in \mathcal{I}} }\mathbbm{ e }^{ 2 \pi \mathbbm{ i } j h_\ell/N } \, c_{ (h_\ell, \k_\ell') }(f) &\text{if there exists $ \k \in \mathcal{I} $ with $ \omega \equiv \k_\ell' \cdot \z_\ell' \imod{M} $}, \\
			0 & \text{otherwise}.
		\end{cases}
	\end{equation*}
    Moreover, defining the matrix $ \mathbf{A}^\ell = \left( \left(\F_M \, \a_j^\ell \right)_{ \omega }\right)_{ j \in [N], \omega \in [M] } $, we have
	\begin{equation*}
		\left( \F_N \, \mathbf{A}^\ell\right)_{ k_\ell \bmod N, \k_\ell' \cdot \z_\ell' \bmod M } = c_\k(f) \text{ for all } \k \in \mathcal{I},
	\end{equation*}
	and the remaining entries of the matrix $ \F_N \, \mathbf{A}^\ell \in \C^{N\times M} $ are zero.
\end{lem}
\begin{proof}
	Using the Fourier series representation of $ f $, we have
	\begin{equation*}
		a_j^\ell(t) := \sum_{ \k \in \mathcal{I} } c_\k(f) \, \mathbbm{ e }^{ 2\pi \mathbbm{ i } \left( \frac{j k_\ell}{N} +   \k_\ell' \cdot \z_\ell't\right) }.
	\end{equation*}
	We calculate for $ \omega \in [M] $ 
	\begin{align*}
		\left( \F_M \, \a_j^\ell \right)_{ \omega }
			&= \frac{1}{M} \sum_{ i \in [M] } \sum_{ \h \in \mathcal{I} } \mathbbm{ e }^{ \frac{ 2 \pi \mathbbm{ i } j h_\ell }{N} } \, c_\h(f) \, \mathbbm{ e }^{ \frac{ 2 \pi \mathbbm{ i } (\h_\ell' \cdot \z_\ell' - \omega) i }{M} }\\
			&= \sum_{ \h \in \mathcal{I} } \mathbbm{ e }^{ \frac{2 \pi \mathbbm{ i } j h_\ell}{N} } \, c_\h(f) \, \delta_{ 0, (\h_\ell' \cdot \z_\ell' - \omega \bmod M)} \\
			&= \sum_{\substack{h_\ell \in \mathcal{B}_N \text{ s.t.}\\(h_\ell, \k_\ell') \in \mathcal{I}}}\mathbbm{ e }^{ \frac{2 \pi \mathbbm{ i } j h_\ell}{N} } \, c_{ (h_\ell, \k_\ell') }(f),
	\end{align*}
	when $ \k \in \mathcal{I} $ is such that $ \k_\ell' \cdot \z_\ell' \equiv \omega \imod{M} $, and clearly is zero when no such $ \k \in \mathcal{I} $ exists.
	Note that the final equality uses that $ \Lambda(\z, M) $ is a reconstructing rank-1 lattice for $ \{\k - k_\ell \mathbf{ e }_\ell \mid \k \in \mathcal{I} \} $.
	Applying $ \F_N $ to $ \mathbf{A}^\ell $ then gives
	\begin{align*}
		\left( \F_N \, \mathbf{A}^\ell \right)_{ k_\ell \bmod N,\, \k_\ell' \cdot \z_\ell' \bmod M }
			&= \frac{1}{N} \sum_{ j \in [N] } \sum_{ h_\ell \in \mathcal{B}_N }c_{ (h_\ell, \k_\ell') }(f) \, \mathbbm{ e }^{ \frac{2 \pi \mathbbm{ i } (h_\ell - k_\ell \bmod N)j}{N} } = c_\k(f).
	\end{align*}

\end{proof}

\begin{algorithm}
	\caption{Frequency Index Recovery by Two Dimensional DFT}
	\label{alg:2dDFT}
	\begin{algorithmic}[1]
		\REQUIRE A multivariate periodic function $ f \in \mathcal{W}(\mathbbm{ T }^d) \cap C(\mathbbm{ T }^d) $ (from which we are able to obtain potentially noisy samples), a multivariate frequency set $ \mathcal{I} \subset \mathcal{B}_N^d $, a rank-1 lattice $ \Lambda(\z, M) $ which is reconstructing for $ \mathcal{I} $ and  $\left\{ \k - k_\ell \mathbf{ e }_\ell \mid \k \in \mathcal{I} \right\} $ for all $ \ell \in [d] $, and an SFT algorithm $ \mathcal{A}_{ s, M } $.
		\ENSURE Sparse coefficient vector $ \mathbf{b} = (b_\k)_{ \k \in \mathcal{B}_N^d } $ (optionally supported on $ \mathcal{I} $, see Line~\ref{alg:2dDFT:Check}), an approximation to $ \left( c\rvert_{ \mathcal{I} } \right)_s^\mathrm{opt} $.
		\STATE Apply $ \mathcal{A}_{ s, M} $ to the univariate restriction of $ f $ to the lattice, $ a(t) := f(t \z) $, to produce $ \v := \mathcal{A}_{ s, M } a $, a sparse approximation of $ \F_M \, \a \in \C^M $. \label{alg:2dDFT:LatticeSFT}
		\FORALL{$\ell \in [d]$} \label{alg:2dDFT:BeginSFTFFTCombo}
			\FORALL{$j \in [N]$}
			\STATE Apply $ \mathcal{A}_{ s, M } $ to $ a_j^\ell(t) := f( \frac{j}{N}, t \z_\ell') $ to produce $ \v_j^\ell := \mathcal{A}_{ s, M } a_j^\ell$, a sparse approximation of $ \F_M \, \a_j^\ell $. \label{alg:2dDFT:ShiftedLatticeSFT}
				\STATE Row $ j $ of $ \mathbf{V}^\ell \gets \v_j^\ell $.
			\ENDFOR
			\FORALL{nonzero columns $ \omega $ of $ \mathbf{V}^\ell $}
				\STATE Apply $ \F_N \, $ to column $ \omega $ of $ \mathbf{V}^\ell $ to produce $ \F_N \, \mathbf{V}^\ell $. \label{alg:2dDFT:ColumnDFT}
			\ENDFOR
		\ENDFOR \label{alg:2dDFT:EndSFTFFTCombo}
		\STATE $ \mathbf{b} \gets \mathbf{0} $  
		\FORALL{$\omega \in \supp (\v)$}
			\FORALL{$\ell \in [d]$}
			\STATE$ ((k_\omega)_\ell, \sim) \gets \arg \min \{ |v_\omega - (\F_N \, \mathbf{V}^\ell)_{ h, \omega' }| \mid (h, \omega') \in \mathcal{B}_N \times [M] \text{ with } h z_\ell + \omega' \equiv \omega \imod{M} \} $  \label{alg:2dDFT:Match}
			\ENDFOR
			\IF{$\k_\omega \cdot \z \equiv \omega \imod{M}$ (and optionally $ \k_\omega \in \mathcal{I}$)} \label{alg:2dDFT:Check}
				\STATE $ b_{ \k_\omega } \gets b_{ \k_\omega } + v_\omega $ 
			\ENDIF
		\ENDFOR
	\end{algorithmic}
\end{algorithm}

\begin{rem}
	We bring special attention to the fact that Algorithm~\ref{alg:2dDFT} requires as input a rank-1 lattice $ \Lambda(\z, M) $ which is reconstructing for not only $ \mathcal{I} $, but also the projections of $ \mathcal{I} $ of the form $ \{\k - k_\ell \mid \k \in \mathcal{I} \} $ for any $ \ell \in [d] $.
	For frequency sets $ \mathcal{I} $ which are downward closed, that is, if $ \mathcal{I} $ is such that for any $ \k \in \mathcal{I} $ and $ \h \in \mathbbm{ Z }^d $, $ |\h| \leq |\k| $ component-wise implies that $ \h \in \mathcal{I} $, any reconstructing rank-1 lattice for $ \mathcal{I} $ is necessarily one for the considered projections as well.
	Thus, for many frequency spaces of interest, e.g., hyperbolic crosses (cf.\ Remarks~\ref{rem:AssignOnlyToI} and \ref{rem:Larger1dBandwidth} as well as Section~\ref{sec:numerics} below), any reconstructing rank-1 lattice for $ \mathcal{I} $ will suffice as input to Algorithm~\ref{alg:2dDFT}.
\end{rem}

\subsubsection{Analysis of Algorithm~\ref{alg:2dDFT}}

\begin{lem}[General recovery result for Algorithm~\ref{alg:2dDFT}.]
	\label{lem:2dDFTRecovery}
	Let $ f $, $ \mathcal{I} $, and $ \Lambda(\z, M) $ be as specified in the input to Algorithm~\ref{alg:2dDFT}.
	Additionally, let $ \mathcal{A}_{ s, M } $ be a noise-robust SFT algorithm satisfying the same constraints as in Lemma~\ref{lem:PhaseEncRecovery} with parameters $ \tau $ and $ \eta_\infty $ holding uniformly for each SFT performed in Algorithm~\ref{alg:2dDFT}.

	Collect the $ \tau $-significant frequencies of $ f $ into the set $ \mathcal{S}_\tau := \{ \k \in \mathcal{I} \mid |c_\k (f) | > \tau \} $ and assume that $ |\supp(c) \cap \mathcal{S}_\tau| \leq s $.
	Then Algorithm~\ref{alg:2dDFT} (ignoring the optional check on Line~\ref{alg:2dDFT:Check}) will produce an $ s $-sparse approximation of the Fourier coefficients of $ f $ satisfying the error estimate
	\begin{equation*}
		\left\| \mathbf{b} - c \right\|_{ \ell^2( \mathbbm{ Z }^d )} \leq \eta_2 + (4\tau + \eta_\infty) \sqrt{ \max(s - | \mathcal{S}_{ 4\tau } |, 0) } + \|c\rvert_{ \mathcal{I} } - c\rvert_{ \mathcal{S}_{ 4\tau } }\|_{ \ell^2( \mathbbm{ Z }^d ) } + \| c - c\rvert_{ \mathcal{I} }\|_{ \ell^2( \mathbbm{ Z }^d) },
	\end{equation*}
	requiring $ \mathcal{O} \left( d N \cdot P(s, M) \right) $ total evaluations of $ f $, in $ \mathcal{O}\left(d N(R(s, M) + s N \log N) \right) $ total operations.
\end{lem}

\begin{proof}
	We begin by assuming that $ f $ is a trigonometric polynomial with $ \supp (c(f)) \subset \mathcal{I} $.
	Since $ \Lambda(\z, M) $ is a reconstructing rank-1 lattice for $ \mathcal{I} $, the DFT-aliasing ensures that Line~\ref{alg:2dDFT:LatticeSFT} of Algorithm~\ref{alg:2dDFT} will return approximate coefficients uniquely corresponding to all $ \tau $-significant frequencies $ \k \in \mathcal{S}_\tau $ which we can label $ v_{ \k \cdot \z \bmod M } $.
	Additionally, Line~\ref{alg:2dDFT:ShiftedLatticeSFT} recovers approximations to all $ \tau $-significant frequencies of $ \F_M \, \a_j^\ell $ which have the form given in Lemma~\ref{lem:2d4supportFinding}.
	In particular, if $ \k \in \mathcal{S}_\tau $, we have
	\begin{align*}
		\tau < |c_\k (f)|
			&= \left| \left( \F_N \, \mathbf{A}^\ell \right)_{ k_\ell \bmod N, \k_\ell' \cdot \z_\ell' \bmod M } \right|\\
			&= \left| \frac{1}{N} \sum_{ j \in [N] } \left( \F_M \, \a_j^\ell \right)_{ \k_\ell' \cdot \z_\ell' \bmod M } \mathbbm{ e }^{ \frac{ -2 \pi \mathbbm{ i } j k_\ell \bmod N }{N} } \right|\\
			&\leq \frac{1}{N} \sum_{ j \in [N] } \left| \left( \F_M \, \a_j^\ell \right)_{ \k_\ell' \cdot \z_\ell' \bmod M } \right| \\
			&\leq \max_{ j \in [N] } \left| (\F_M \, \a_j^\ell)_{ \k_\ell' \cdot \z_\ell' \bmod M } \right|.
	\end{align*}
	Thus, there exists at least one $ \F_M \, \a_j^\ell $ with $ \k_\ell' \cdot \z_\ell' \bmod M $ recovered as a $ \tau $-significant frequency in the SFT of Line~\ref{alg:2dDFT:ShiftedLatticeSFT}, and $ \k_\ell' \cdot \z_\ell' \bmod M $ will be a nonzero column in $ \mathbf{V}^\ell $ for all $ \k \in \mathcal{S}_\tau $.

	Analyzing these SFTs in more detail for any $ \k \in \mathcal{I} $ such that $ \k_\ell' \cdot \z_\ell' \bmod M $ is a nonzero column of $ \mathbf{V}^\ell $, we write
	\begin{equation*}
		\left( \v_{ j }^\ell \right)_{ \k_\ell' \cdot \z_\ell' \bmod M } = \left( \F_M \, \a_j^\ell \right)_{ \k_\ell' \cdot \z_\ell' \bmod M } + \left( \eta_j^\ell \right)_{ \k_\ell' \cdot \z_\ell' \bmod M }
	\end{equation*}
	where, by the $ \ell^\infty $ and recovery guarantees for $ \mathcal{A}_{ s, M } $, the error satisfies
	\begin{equation*}
		\left| \left( \eta_j^\ell \right)_{ \k_\ell' \cdot \z_\ell' \bmod M } \right| \leq
			\begin{cases}
				\eta_\infty &\text{if } \left( \v_j^\ell \right)_{ \k_\ell' \cdot \z_\ell' \bmod M } \neq 0\\
				\tau &\text{if } \left( \v_j^\ell \right)_{ \k_\ell' \cdot \z_\ell' \bmod M } = 0
			\end{cases}
			\leq \tau.
	\end{equation*}
	Thus, in the application of $ \F_N $ to column $ \k_\ell' \cdot \z_\ell' \bmod M $ of $ \mathbf{V}^\ell $, we have
	\begin{align*}
		\left( \F_N \, \mathbf{V}^\ell \right)&_{ k_\ell \bmod N, \k_\ell' \cdot \z_\ell' \bmod M } \\
			&= \left( \F_N \, \mathbf{A}^\ell \right)_{ k_\ell \bmod N, \k_\ell' \cdot \z_\ell' \bmod M } + \left( \F_N \, \left( \left( \eta_j^\ell \right)_{ \k_\ell' \cdot \z_\ell' \bmod M }  \right)_{ j \in [N] }  \right)_{ k_\ell \bmod N } \\
			&=: c_\k(f) + \eta_\k^\ell
	\end{align*}
	with
	\begin{equation*}
		|\eta_\k^\ell|
			= \left| \frac{1}{N} \sum_{ j \in [N] } \left( \eta_j^\ell \right)_{ \k_\ell' \cdot \z_\ell' \bmod M } \mathbbm{ e }^{ \frac{- 2 \pi \mathbbm{ i } j k_\ell \bmod N}{N}  } \right| 
			\leq \max_{ j \in [N] } \left| \left( \eta_j^\ell \right)_{ \k_\ell' \cdot \z_\ell' \bmod M } \right| \leq \tau.
	\end{equation*}
	These same calculations apply to the computed columns of $ \F_N \, \mathbf{V}^\ell $ which do not correspond to values of $ \k_\ell' \cdot \z_\ell' \bmod M $ for some $ \k \in \mathcal{I} $ since we assume $ \supp(c(f)) \subset \mathcal{I} $, and so at worst, these columns are filled with noise bounded in magnitude by $ \tau $.

	Restricting our attention to $ \k \in \mathcal{S}_{ 4\tau } \subset \mathcal{S}_{ \tau }$, we know that Line~\ref{alg:2dDFT:Match} will be run with $ \omega = \k \cdot \z \bmod M $ and $ (k_\ell \bmod N, \k_\ell' \cdot \z_\ell' \bmod M) $ as an admissible index in the minimization.
	By the reconstructing property of $ \Lambda(\z, M) $, no other $ \h \in \mathcal{I} $ will correspond to an admissible index $ (h_\ell \bmod N, \h_\ell' \cdot \z_\ell' \bmod M) $, and so the only remaining values of $ (\F_N \, \mathbf{V}^\ell)_{ h, \omega' } $ in the minimization correspond to pure noise $ \eta $ bounded in magnitude by $ \tau $.
	Analyzing the objective at $ (k_\ell \bmod N, \k_\ell' \cdot \z_\ell' \bmod M)  $, we find
	\begin{equation*}
		| v_{ \k \cdot \z \bmod M } - ( \F_N \, \mathbf{V}^\ell)_{ k_\ell \bmod N, \k_\ell' \cdot \z_\ell' \bmod M } | \leq 2\tau < |c_\k(f)| - 2\tau \leq | v_{ \k \cdot \z \bmod M } - \eta|,
	\end{equation*}
	and so the value for $ (k_\omega)_\ell $ will in fact be assigned $ k_\ell $.
	Thus, after all $ d $ components of $ \k_\omega = \k $ have been recovered, $ b_\k $ will be assigned $ v_{ \k \cdot \z \bmod M } $.

	The remaining $ \max(s - | \mathcal{S}_{ 4\tau } |, 0) $ nonzero entries of $ \v $ can be distributed to entries of $ \mathbf{b} $ possibly correctly but with no guarantee; at the very least however, these values must be at most $ 4\tau + \eta_\infty $ in magnitude.
	We split $ \mathbf{b} $ as $ \mathbf{b} =  \mathbf{b}^\mathrm{correct} + \mathbf{b}^\mathrm{incorrect} $ to account for the values of $ \v $ respectively assigned correctly and incorrectly and note that $ \supp (\mathbf{b}^\mathrm{correct})  \supset \mathcal{S}_{ 4\tau } $.
	We then estimate the error as
	\begin{align*}
		\| \mathbf{b} - c \|_{ \ell^2( \mathbbm{ Z }^d) }
			&\leq \| \mathbf{b}^\mathrm{correct} - c \rvert_{ \supp (\mathbf{b}^\mathrm{correct}) } \|_{ \ell^2( \mathbbm{ Z }^d ) } + \| \mathbf{b}^\mathrm{incorrect} \|_{ \ell^2( \mathbbm{ Z^d } ) } + \| c - c\rvert_{ \supp (\mathbf{b}^\mathrm{correct}) } \|_{ \ell^2 (\mathbbm{ Z }^d) }  \\
			&\leq \eta_2 + (4\tau + \eta_\infty) \sqrt{ \max(s - | \mathcal{S}_{ 4\tau } |, 0) }+ \|c - c \rvert_{ \mathcal{S}_{ 4\tau } } \|_{ \ell^2( \mathbbm{Z}^d ) }.
	\end{align*}
	As in the proof of Lemma~\ref{lem:PhaseEncRecovery}, we note that the mandatory check in Line~\ref{alg:2dDFT:Check} helps ensure that all misassigned values $ v_\omega $ which contribute to $ \mathbf{b}^\mathrm{incorrect} $ correspond to reconstructed $ \k_\omega $ \emph{outside} of $ \mathcal{I} $, with the optional check in this line (see Remark~\ref{rem:AssignOnlyToI}) eliminating $ \mathbf{b}^\mathrm{incorrect} $ and the corresponding term in the error estimate entirely.

	Now, supposing that the Fourier support of $ f $ is not limited to only $ \mathcal{I} $, just as in the analysis for Algorithm~\ref{alg:PhaseEnc}, we treat $ f $ as a perturbation of $ f_{ \mathcal{I} } $, and use the robust SFT algorithm and the previous argument to approximate $ c\rvert_{ \mathcal{I} } $.
	Note again that in each SFT, the noise added when using measurements of $ f $ as proxies for those of $ f_{ \mathcal{I} } $ is compounded by $ \| f_{ \mathbbm{ Z }^d \setminus \mathcal{I} } \|_\infty $ and is bounded by $ \| c - c\rvert_{ \mathcal{I} } \|_{ \ell^1( \mathbbm{ Z }^d) } $.
	Applying the guarantees above gives
	\begin{align*}
		\| \mathbf{b} - c \|_{ \ell^2( \mathbbm{ Z }^d )} 
			&\leq \| \mathbf{b} - c\rvert_{ \mathcal{I} } \|_{ \ell^2( \mathbbm{ Z }^d) } + \| c - c\rvert_{ \mathcal{I} }\|_{ \ell^2(\mathbbm{Z}^d )}\\
			&\leq \eta_2 + (4\tau + \eta_\infty) \sqrt{ \max(s - | \mathcal{S}_{ 4\tau }|, 0) } + \| c\rvert_{ \mathcal{I} } - c\rvert_{ \mathcal{S}_{ 4\tau } }\|_{ \ell^2( \mathbbm{ Z }^d) } + \| c - c\rvert_{ \mathcal{I} }\|_{ \ell^2( \mathbbm{ Z }^d) }.
	\end{align*}

	Employing fast Fourier transforms for the at most $ dsN $ DFTs, the computational complexity of Lines~\ref{alg:2dDFT:BeginSFTFFTCombo}--\ref{alg:2dDFT:EndSFTFFTCombo} is $ \mathcal{O}\left( d(N \cdot R(s, M) + s N^2 \log N) \right) $  (which dominates the complexity of the remainder of the algorithm).
	Since $ 1 + dN $ SFTs are required, the number of $ f $ evaluations is $ \mathcal{O}(dN \cdot P(s, M)) $.
\end{proof}

\begin{rem}
	Though the number of nonzero columns of $ \mathbf{V}^\ell $ can be theoretically at most $ sN $, in practice with a high quality algorithm, each of the $ N $ SFTs should recover nearly the same frequencies, meaning that there are actually $ \mathcal{O}(s) $ columns.
	This would remove a power of $ N $ in the second term of the runtime estimate.

	Note however, that even with near exact SFT algorithms, recovering exactly $ s $ total frequencies is not a certainty.
	There can be cancellations for certain terms in $ \F_M \, \a_j^\ell $ depending interactions between the coefficients sharing the same values on their $ [d] \setminus \{\ell\} $ entries, which makes it possible that an SFT on $ \F_M \, \a_j^\ell $ will miss coefficients.
	If required to output $ s $-entries, an SFT algorithm could favor some noisy value corresponding to a frequency outside the support.
\end{rem}

\begin{rem}
	Though we perform an exact FFT of the nonzero columns of $ \mathbf{V}^\ell $ in Line~\ref{alg:2dDFT:ColumnDFT} of Algorithm~\ref{alg:2dDFT}, Lemma~\ref{lem:2d4supportFinding} implies that the resulting matrix will be as sparse as the original function's Fourier transform.
	Thus, for a truly compressible function, an SFT down the columns of $ \mathbf{V}^\ell $ would be feasible as well.
	However, in especially higher dimensions, even small $ N $ can support large frequency spaces $ \mathcal{I} $.
	In these large frequency spaces, what is perceived as relatively sparse can therefore quickly surpass $ N $, rendering an $ s $-sparse, length $ N $ SFT useless.
\end{rem}

Applying the discrete sublinear-time SFT from Theorem~\ref{thm:1dDiscreteSFT} to Lemma~\ref{lem:2dDFTRecovery} analogously to the derivation of Corollary~\ref{cor:PhaseEncRecoveryDiscreteSFT} from Lemma~\ref{lem:PhaseEncRecovery} allows for the following recovery bound for Algorithm~\ref{alg:2dDFT}.
In particular, we observe asymptotically improved error guarantees over Corollary~\ref{cor:PhaseEncRecoveryDiscreteSFT} at the cost of a slight increase in runtime.

\begin{cor}[Algorithm \ref{alg:2dDFT} with discrete sublinear-time SFT]
	\label{cor:2dDFTRecoveryDiscreteSFT}
	For $ \mathcal{I} \subset \mathbbm{ Z }^d $ with reconstructing rank-1 lattice $ \Lambda(\z, M) $ and the function $ f \in \mathcal{W}(\mathbbm{ T }^d) \cap C( \mathbbm{ T }^d ) $, we consider applying Algorithm~\ref{alg:2dDFT} where each function sample may be corrupted by noise at most $ e_\infty \geq 0 $ in absolute magnitude. Using the discrete sublinear-time SFT algorithm $ \mathcal{A}_{ 2s, M }^\mathrm{disc} $ or $ \mathcal{A}_{ 2s, M }^\mathrm{disc, MC}  $ with parameter $ 1 \leq r \leq \frac{M}{36} $ will produce $ \mathbf{b} = (b_{ \k })_{ \k \in \mathcal{B}_N^d } $ a $ 2s $-sparse approximation of $ c $ satisfying the error estimate
	\begin{align*}
		\left\| \mathbf{b} - c \right\|_{ \ell^2( \mathbbm{ Z }^d )} 
		&\leq 206 \frac{\|c\rvert_{ \mathcal{I} } - (c\rvert_{ \mathcal{I} })_s^\mathrm{opt}\|_1}{\sqrt{ s }} + 820 \sqrt{ s } (\|f\|_\infty M^{ -r } + \|c - c\rvert_{ \mathcal{I} }\|_1 + e_\infty ) \\
		&\qquad + \|c - c\rvert_{ \mathcal{I} }\|_2
	\end{align*}
	albeit with probability $ 1 - \sigma \in [0, 1) $ for the Monte Carlo version.

	The total number of evaluations of $ f $ and the computational complexity will be 
	\begin{equation*}
		\mathcal{O} \left( dsN \left( \frac{s r^{ 3/2 } \log^{ 11/2 }M}{\log s} + N \log N \right) \right) \text{ or } \mathcal{O} \left( dsN \left( r^{ 3/2 } \log^{ 9/2 }(M) \log \left( \frac{dNM}{\sigma} \right) + N \log N \right) \right)
	\end{equation*}
	for $ \mathcal{A}_{ 2s, M }^\mathrm{disc} $ or $ \mathcal{A}_{ 2s, M }^\mathrm{disc, MC} $ respectively.
\end{cor}
\begin{proof}
	In detail, let
	\begin{gather*}
		A = \frac{\|c\rvert_{ \mathcal{I} } - (c\rvert_{ \mathcal{I} })_s^\mathrm{opt}\|_1}{\sqrt{ s }}, \quad B = \left( 274 + 96 \sqrt{ 2 } \right) \\
		\delta = \left( \frac{A}{2\sqrt{ s }} + 2 (\|f\|_\infty M^{ -r } + \|c - c\rvert_{ \mathcal{I} } \|_1 + e_\infty) \right).
	\end{gather*}
	Then
	\begin{gather*}
		\tau \leq 12(1 + \sqrt{ 2 }) \delta, \quad \eta_\infty \leq 3 \sqrt{ 2 } \delta\\
		\eta_2 \leq \frac{A}{2} + 76 \sqrt{ s } \delta, \\
		\|c\rvert_{ \mathcal{I} } - c\rvert_{ \mathcal{S}_{ 4\tau } }\|_2 \leq \frac{A}{2} + 4\sqrt{ 2 } \tau \sqrt{ s } = \frac{A}{2} + 48 (2 + \sqrt{ 2 }) \delta \sqrt{ s }.
	\end{gather*}
	Our error bound is then
	\begin{align*}
		\left\| \mathbf{b} - c \right\|_{ \ell^2( \mathbbm{ Z }^d )}
		&\leq \eta_2 + (4\tau + \eta_\infty) \sqrt{ 2s } + \|c\rvert_{ \mathcal{I} } - c\rvert_{ \mathcal{S}_{ 4\tau } }\|_{ \ell^2( \mathbbm{ Z }^d ) } + \| c - c\rvert_{ \mathcal{I} }\|_{ \ell^2( \mathbbm{ Z }^d) }\\
		&\leq A + B\sqrt{ s }\delta + \| c - c\rvert_{ \mathcal{I} }\|_2\\
		&= \left(1 + \frac{B}{2}\right) A + 2B \sqrt{ s } (\|f\|_\infty M^{ -r } + \|c - c\rvert_{ \mathcal{I} }\|_1 + e_\infty) + \| c - c\rvert_{ \mathcal{I} }\|_2\\
		&\leq 206\frac{\|c\rvert_{ \mathcal{I} } - (c\rvert_{ \mathcal{I} })_s^\mathrm{opt}\|_1}{\sqrt{ s }}  + 820 \sqrt{ s } ( \|f\|_\infty M^{ -r } + \|c - c\rvert_{ \mathcal{I} }\|_1 + e_\infty) \\
		&\qquad+ \|c - c\rvert_{ \mathcal{I} }\|_2.
	\end{align*}
\end{proof}

Again, the same strategy from Corollary~\ref{cor:PhaseEncRecoverySublinearSFT} of widening the frequency band and shifting the one-dimensional transforms accordingly allows us to use the nonequispaced SFT algorithm from Theorem~\ref{thm:1dSFT} in Algorithm~\ref{alg:2dDFT}.
Note here that the widening and shifting occurs on a dimension by dimension basis so as to account for the differing one-dimensional frequencies of the form $ \k_\ell' \cdot \z_\ell' $ for $ \k \in \mathcal{I} $.

\begin{cor}[Algorithm \ref{alg:2dDFT} with nonequispaced sublinear-time SFT]
	\label{cor:2dDFTRecoverySublinearSFT}
	For $ \mathcal{I} \subset \mathcal{B}_N^d $, let $ \tilde M $ be the larger one-dimensional bandwidth parameter from Corollary~\ref{cor:PhaseEncRecoverySublinearSFT}, and additionally define $ \tilde M^\ell := 2 \max_{ \k \in \mathcal{I} } |\k_\ell' \cdot \z_\ell'| + 1 $.
For $ \Lambda(\z, M) $, a reconstructing rank-1 lattice for $ \mathcal{I} $ with $ M \leq \min\{\tilde M, \min_{ \ell \in [d] } \tilde M^\ell \} $, and the function $ f \in \mathcal{W}(\mathbbm{ T }^d) \cap C( \mathbbm{ T }^d ) $, we consider applying Algorithm~\ref{alg:2dDFT} where each function sample may be corrupted by noise at most $ e_\infty \geq 0 $ in absolute magnitude with the following modifications:
	\begin{enumerate}
		\item use the sublinear-time SFT algorithm $ \mathcal{A}_{ 2s, \tilde M }^\mathrm{sub} $ or $ \mathcal{A}_{ 2s, \tilde M }^\mathrm{sub, MC} $ in Line~\ref{alg:2dDFT:LatticeSFT} and $ \mathcal{A}_{ 2s, \tilde M^\ell }^\mathrm{sub} $ or $ \mathcal{A}_{ 2s, \tilde M^\ell }^\mathrm{sub, MC} $ in Line~\ref{alg:2dDFT:ShiftedLatticeSFT}
		\item and only check equality against $ \omega $ in Line~\ref{alg:2dDFT:Match} (rather than equivalence modulo $ M $),
	\end{enumerate}
	to produce $ \mathbf{b} = (b_\k)_{ \k \in \mathcal{B}_N^d } $ a $ 2s $-sparse approximation of $ c $ satisfying the error estimate
	\begin{align*}
		\left\| \mathbf{b} - c \right\|_{ \ell^2( \mathbbm{ Z }^d )} 
		&= 98 \left(\frac{\|c\rvert_{ \mathcal{I} } - (c\rvert_{ \mathcal{I} })_s^\mathrm{opt}\|_1}{\sqrt{ s }} + \sqrt{ s } \|c - c\rvert_{ \mathcal{I} }\|_1 + \sqrt{ s } e_\infty\right) + \| c - c\rvert_{ \mathcal{I} }\|_2 \\
	\end{align*}
	albeit with probability $ 1 - \sigma \in [0, 1) $ for the Monte Carlo version.

	Letting $ \bar M = \max( \tilde M, \max_{ \ell \in [d] } \tilde M^\ell) $, the total number of evaluations of $ f $ will be 
	\begin{equation*}
		\mathcal{O} \left( \frac{dNs^2 \log^4 \bar M}{\log s}  \right) \text{ or } \mathcal{O} \left( dNs \log^3 \bar M \log \left( \frac{dN\bar M}{\sigma} \right) \right)
	\end{equation*}
	with associated computational complexities
	\begin{equation*}
		\mathcal{O} \left( dNs \left( \frac{s \log^4 \bar M}{\log s} + N \log N \right)  \right) \text{ or } \mathcal{O} \left( dNs \left( \log^3 \bar M \log \left( \frac{dN\bar M}{\sigma} \right) + N \log N \right) \right)
	\end{equation*}
	for $ \mathcal{A}_{ 2s, \cdot }^\mathrm{sub} $ and $ \mathcal{A}_{ 2s, \cdot }^\mathrm{sub, MC} $ respectively.
\end{cor}
\begin{proof}
	In detail, let
	\begin{gather*}
		A = \frac{\|c\rvert_{ \mathcal{I} } - (c\rvert_{ \mathcal{I} })_s^\mathrm{opt}\|_1}{\sqrt{ s }}, \quad B = 40 (1 + \sqrt{ 2 })\\
		\delta = \left( \frac{A}{\sqrt{ s }} + \|c - c\rvert_{ \mathcal{I} } \|_1 + e_\infty \right).
	\end{gather*}
	Then
	\begin{gather*}
		\tau \leq (4 + 2\sqrt{ 2 }) \delta, \quad \eta_\infty \leq \sqrt{ 2 } \delta\\
		\eta_2 \leq \frac{A}{2} + (8 \sqrt{ 2 } + 6) \sqrt{ s } \delta, \\
		\|c\rvert_{ \mathcal{I} } - c\rvert_{ \mathcal{S}_{ 4\tau } }\|_2 \leq \frac{A}{2} + 4\sqrt{ 2 }\tau \sqrt{ s } = \frac{A}{2} + 16(1 + \sqrt{ 2 })\delta \sqrt{ s }.
	\end{gather*}
	Our error bound is then
	\begin{align*}
		\|\mathbf{b} - c\|_2 
		&\leq \eta_2 + (4\tau + \eta_\infty) \sqrt{ 2s } + \|c\rvert_{ \mathcal{I} } - c\rvert_{ \mathcal{S}_{ 4\tau } }\|_{ \ell^2( \mathbbm{ Z }^d ) } + \| c - c\rvert_{ \mathcal{I} }\|_{ \ell^2( \mathbbm{ Z }^d) }\\
		&= A + B\sqrt{ s }\delta + \| c - c\rvert_{ \mathcal{I} }\|_2 \\
		&= (1 + B) \frac{\|c\rvert_{ \mathcal{I} } - (c\rvert_{ \mathcal{I} })_s^\mathrm{opt}\|_1}{\sqrt{ s }} + B \sqrt{ s } \|c - c\rvert_{ \mathcal{I} }\|_1 + B \sqrt{ s } e_\infty + \| c - c\rvert_{ \mathcal{I} }\|_2 \\
		&= 98 \left(\frac{\|c\rvert_{ \mathcal{I} } - (c\rvert_{ \mathcal{I} })_s^\mathrm{opt}\|_1}{\sqrt{ s }} + \sqrt{ s } \|c - c\rvert_{ \mathcal{I} }\|_1 + \sqrt{ s } e_\infty\right) + \| c - c\rvert_{ \mathcal{I} }\|_2 \\
	\end{align*}
\end{proof}

\begin{rem}
	The bounds in Remark~\ref{rem:Larger1dBandwidth} will still hold for $ \tilde M^\ell $ as well; thus one of these upper bounds can be used as the effective bandwidth parameter for every SFT without having to calculate the $ d + 1$ bandwidths by scanning $ \mathcal{I} $.
	Again however, if this scan is tolerable, one can reduce the overall complexity by using analogous minimal bandwidths discussed in Remark~\ref{rem:Larger1dBandwidth} along with corresponding frequency shifts.
\end{rem}

\section{Numerics}%
\label{sec:numerics}
We now demonstrate the effectiveness of our phase encoding and two-dimensional DFT algorithms for computing Fourier coefficients of multivariate functions in a series of empirical tests.
The two techniques are implemented in MATLAB, with the code for the algorithms and tests in this section publicly available\footnote{available at \url{https://gitlab.com/grosscra/Rank1LatticeSparseFourier}}.
The results below use a MATLAB implementation\footnote{available at \url{https://gitlab.com/grosscra/SublinearSparseFourierMATLAB}} of the randomized univariate sublinear-time nonequispaced algorithm $ \mathcal{A}_{ 2s, M }^\mathrm{sub,MC} $ (cf.\ Theorem~\ref{thm:1dSFT}) as the underlying SFT for both multivariate approaches as this allows for the fastest runtime and most sample efficient implementations.

In the univariate code, all parameters but one are qualitatively tuned below theoretical upper bounds to increase efficiency while maintaining accuracy and are kept constant between tests below.
In particular, we fix the values $ \texttt{C} := 1 $, $ \texttt{sigma} := 2/3 $, and $ \texttt{primeShift} := 0 $ (see the documentation and the original paper \cite{Iwen13} for more detail).
The only parameter we vary is \texttt{randomScale} which affects the rate at which the deterministic algorithm $ \mathcal{A}_{ 2s, M }^\mathrm{sub} $ is randomly sampled to produce the Monte Carlo version $ \mathcal{A}_{ 2s, M }^\mathrm{ sub, MC} $.
This parameter represents a multiplicative scaling on logarithmic factors of the bandwidth which determines how many prime numbers are randomly selected from those used in the deterministic SFT implementation.
Therefore, lower values of \texttt{randomScale} will result in using fewer prime numbers, decreasing the number of function samples and overall runtime at the risk of a higher probability of failure.
We consider values well below the code default and theoretical upper bound of $ 21 $ given in \cite{Iwen13}.

\subsection{Exactly sparse case}

In the beginning, we consider the case of multivariate trigonometric polynomials with frequencies supported within hyperbolic cross index sets. We define the $d$-dimensional hyperbolic cross frequency set
$$\mathcal{H}_N^d := \left\{\k \in \mathbbm{Z}^d \colon \prod_{\ell=1}^d \max(1,|k_\ell|) \leq \frac{N}{2} \quad \text{and} \quad \max_{\ell=1,\ldots,d} k_\ell<\frac{N}{2}
\right\}\subset \mathcal{B}_N^d$$ of expansion $N\in\mathbbm{N}$.
For a given sparsity $s$, we choose $s$ many frequencies uniformly at random from $\mathcal{H}_N^d$, and we randomly draw corresponding Fourier coefficients $c_\k$ from $[-1,1] + \mathbbm{ i }\,[-1,1]$, $|c_\k|\geq 10^{-3}$. For each parameter setting, we perform the tests 100 times and determine the success rate, i.e., the relative number of cases (out of the 100) where all frequencies were correctly detected, as well as the average number of samples.

\FloatBarrier

\subsubsection{Random frequency sets within 10-dimensional hyperbolic cross and high-dimensional full cuboids}

We set the spatial dimension $d:=10$, the expansion $N:=33$, and use $\mathcal{I}:=\mathcal{H}_{33}^{10}$ as set of possible frequencies with cardinality $|\mathcal{I}|=45\,548\,649$.
Then, the rank-1 lattice with generating vector
\begin{equation}\label{equ:numerics_lattice_z_for_H_33_10}
\z:=(1,\, 33,\, 579,\, 3\,628,\, 21\,944,\, 169\,230,\, 1\,105\,193,\, 7\,798\,320,\, 49\,768\,670,\, 320\,144\,128)^\top
\end{equation}
and lattice size $M:=2\,040\,484\,044$ is a reconstructing one. We apply Algorithm~\ref{alg:PhaseEnc} and Algorithm~\ref{alg:2dDFT} with the SFT algorithm $\mathcal{A}_{ 2s, \tilde M }^\mathrm{sub, MC}$.

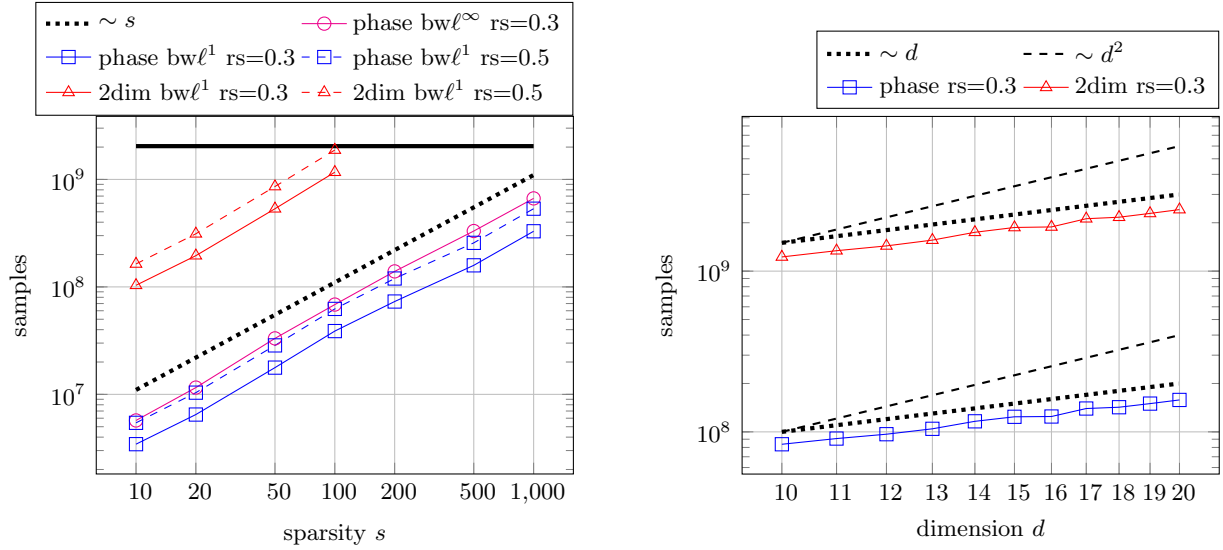
\begin{figure}[!h]
\begin{subfigure}[t]{0.48\textwidth}
		\begin{tikzpicture}[baseline]
		\begin{axis}[
		font=\footnotesize,
		enlarge x limits=true,
		enlarge y limits=true,
		height=0.8\textwidth,
		grid=major,
		width=\textwidth,
		xtick={10,20,50,100,200,500,1000},
		xmode=log,
		ymode=log,
		xticklabel={
			\pgfkeys{/pgf/fpu=true}
			\pgfmathparse{exp(\tick)}%
			\pgfmathprintnumber[fixed relative, precision=3]{\pgfmathresult}
			\pgfkeys{/pgf/fpu=false}
		},
		xlabel={sparsity $s$},
		ylabel={samples},
		legend style={legend cell align=left, at={(1,1.32)}},
		legend columns = 2,
		]
		\addplot[forget plot,black,no marks,ultra thick] coordinates {
          (10,2040484044) (1000,2040484044)
        };
        \addplot[dotted,ultra thick,samples=100,black,domain=10:1000] {1.1e6*x};
        \addlegendentry{$\sim s$}
		\addplot[magenta,mark=o,mark size=2.5pt,mark options={solid}] coordinates {
 (10,5719643.930) (20,11564591.060) (50,33192690.080) (100,68499805.770) (200,139422192.800) (500,331049457.750) (1000,666182351.230) %
		};
	    \addlegendentry{phase bw$\ell^\infty$ rs=0.3}
		\addplot[blue,mark=square,mark size=2.5pt,mark options={solid}] coordinates {
 (10,3440435.350) (20,6490351.120) (50,17737835.500) (100,38807702.340) (200,73104032.540) (500,158743876.390) (1000,329344657.840) %
		};
	    \addlegendentry{phase bw$\ell^1$ rs=0.3}
		\addplot[blue,mark=square,mark size=2.5pt,dashed,mark options={solid}] coordinates {
 (10,5435147.520) (20,10372995.710) (50,28556370.810) (100,62408538.170) (200,119675996.000) (500,256934928.470) (1000,534237652.300) %
		};
	    \addlegendentry{phase bw$\ell^1$ rs=0.5}
		\addplot[red,mark=triangle,mark size=2.5pt,mark options={solid}] coordinates {
 (10,103525827.350) (20,195300565.520) (50,533747595.500) (100,1167759043.140) %
        };
        \addlegendentry{2dim bw$\ell^1$ rs=0.3}
		\addplot[red,mark=triangle,mark size=2.5pt,dashed,mark options={solid}] coordinates {
 (10,163548529.920) (20,312132870.910) (50,859287158.010) (100,1877929648.570) %
        };
        \addlegendentry{2dim bw$\ell^1$ rs=0.5}
		\end{axis}
		\end{tikzpicture}
\subcaption{Samples vs.\ sparsity $s$. Random frequencies are chosen from hyperbolic cross $\mathcal{I}:=\mathcal{H}_{33}^{10}$. ``bw$\ell^\infty$'' and ``bw$\ell^1$'' correspond to bandwidth parameters $\tilde{M}=dNM\approx 6.7\cdot 10^{11}$ and $\tilde{M}=1 + 2 \| \z \|_\infty \max_{ \k \in \mathcal{I} } \| \k \|_1\approx 1.6\cdot 10^{10}$, respectively.}
\label{fig:numerics:trigpoly_rand_hc}
\end{subfigure}
	\hfill
\begin{subfigure}[t]{0.48\textwidth}
		\begin{tikzpicture}[baseline]
		\begin{axis}[
		font=\footnotesize,
		enlarge x limits=true,
		enlarge y limits=true,
		height=0.8\textwidth,
		grid=major,
		width=\textwidth,
		xtick={10,11,12,13,14,15,16,17,18,19,20},
		xmode=log,
		ymode=log,
		xticklabel={
			\pgfkeys{/pgf/fpu=true}
			\pgfmathparse{exp(\tick)}%
			\pgfmathprintnumber[fixed relative, precision=3]{\pgfmathresult}
			\pgfkeys{/pgf/fpu=false}
		},
		xlabel={dimension $d$},
		ylabel={samples},
		legend style={legend cell align=left, at={(1,1.24)}},
		legend columns = 2,
		]
        \addplot[dotted,ultra thick,samples=100,black,domain=10:20] {1e7*x};
        \addlegendentry{$\sim d$}
        \addplot[dashed,thick,samples=100,black,domain=10:20] {1e6*x*x};
        \addlegendentry{$\sim d^2$}
        \addplot[forget plot,dotted,ultra thick,samples=100,black,domain=10:20] {1.5e8*x};
        \addplot[forget plot,dashed,thick,samples=100,black,domain=10:20] {1.5e7*x*x};
		\addplot[blue,mark=square,mark size=2.5pt,mark options={solid}] coordinates {
          (10,83724703.37) (11,90845430.12) (12,96700248.06) (13,104469819.02) (14,116306090.55) (15,124094091.52) (16,124609978.75) (17,139702664.70) (18,142324967.82) (19,149913136.20) (20,157986967.74) %
		};
	    \addlegendentry{phase rs=0.3}
		\addplot[red,mark=triangle,mark size=2.5pt,mark options={solid}] coordinates {
          (10,1225425203.87) (11,1339970094.27) (12,1435626759.66) (13,1559585155.37) (14,1744591358.25) (15,1869167253.52) (16,1883809678.75) (17,2118823747.95) (18,2164837668.42) (19,2286175327.05) (20,2414943649.74) %
        };
        \addlegendentry{2dim rs=0.3}
		\end{axis}
		\end{tikzpicture}
\subcaption{Samples vs.\ spatial dimension $d$. Random frequencies are chosen from \textbf{full cuboid} $\mathcal{I}$ of cardinality $|\mathcal{I}|\approx 10^{12}$ with corresponding lattice of size $M=|\mathcal{I}|$ and bandwidth parameter $\tilde{M}=M$.}
\label{fig:numerics:trigpoly_rand_cuboid}
\end{subfigure}
\caption{Average number of samples over 100 test runs for Algorithm~\ref{alg:PhaseEnc} with $\mathcal{A}_{ 2s, \tilde M }^\mathrm{sub, MC}$, denoted by ``phase'', and  Algorithm~\ref{alg:2dDFT} with $\mathcal{A}_{ 2s, \tilde M }^\mathrm{sub, MC}$, denoted by ``2dim'', on random multivariate trigonometric polynomials, setting $\texttt{randomScale}:=\text{rs}$.}
\end{figure}

In Figure~\ref{fig:numerics:trigpoly_rand_hc}, the average numbers of samples (over 100 test runs) are plotted against the used sparsities $s\in\{10,20,50,100,200,500,1000\}$ for Algorithm~\ref{alg:PhaseEnc} and $s\in\{10,20,50,100\}$ for Algorithm~\ref{alg:2dDFT}.
The magenta line with circles corresponds to Algorithm~\ref{alg:PhaseEnc} with bandwidth parameter $\tilde{M}=dNM\approx 6.7\cdot 10^{11}$ and $\texttt{randomScale}:=0.3$. We observe that the number of samples grow nearly linearly with respect to the sparsity~$s$. Moreover, the success rate is at least 0.99 (99 out of 100 test runs), where we define success such that the support of output (sparse coefficient vector) contains the true frequencies. Next, we reduce the bandwidth $\tilde M$ to $1 + 2 \| \z \|_\infty \max_{ \k \in \mathcal{I} } \| \k \|_1\approx 1.6\cdot 10^{10}$, see also Remark~\ref{rem:Larger1dBandwidth}, and visualize this as solid blue line with squares. This smaller bandwidth causes a decrease in the number of samples of up to 50 percent while only mildly decreasing the success rates to values not below 0.90. Increasing the $\texttt{randomScale}$ parameter to 0.5, denoted by dashed blue line with squares, raises the success rate to 1.00 while achieving still fewer samples than bandwidth parameter $\tilde{M}=dNM\approx 6.7\cdot 10^{11}$ and $\texttt{randomScale}=0.3$ (solid magenta line with circles). The numbers of samples for Algorithm~\ref{alg:2dDFT} are plotted as solid and dashed red lines with triangles for $\texttt{randomScale}:=0.3$ and $0.5$, respectively, choosing the bandwidth $\tilde M:=1 + 2 \| \z \|_\infty \max_{ \k \in \mathcal{I} } \| \k \|_1\approx 1.6\cdot 10^{10}$. We observe that Algorithm~\ref{alg:2dDFT} requires much more samples, more than one order of magnitude, compared to Algorithm~\ref{alg:PhaseEnc}, while achieving similar success rates. For comparison, in case of sparsity $s=100$ and $\texttt{randomScale}=0.5$, Algorithm~\ref{alg:2dDFT} takes almost $M=2\,040\,484\,044$ samples, where the latter would be required by a non-SFT approach which uses all rank-1 lattice nodes.

In Figure~\ref{fig:numerics:trigpoly_rand_cuboid}, we investigate the dependence of the required number of samples of Algorithm~\ref{alg:PhaseEnc} and~\ref{alg:2dDFT} on the spatial dimension~$d$, where we consider the values $d\in\{10,11,\ldots,20\}$. For this, we use a slightly different setting, where we choose $s=100$ random frequencies from a full cuboid of cardinality $\approx 10^{12}$. For instance, we utilize the cuboid $\mathcal{I}:=\{-8,-7,\ldots,7\}^9\times\{-7,-6,\ldots,7\}$, $|\mathcal{I}|\approx 1.03\cdot 10^{12}$, in the case $d=10$ and $\mathcal{I}:=\{-2,-1,\ldots,2\}\times\{-2,-1,0,1\}^{18}\times\{-1,0,1\}$, $|\mathcal{I}|\approx 1.03\cdot 10^{12}$, for $d=20$.
The rank-1 lattice size and the bandwidth parameter are chosen to be $M=\tilde{M}=|\mathcal{I}|$. The generating vector $\z$ follows from the edge lengths of the cube, e.g., $\z:=(1,16,16\cdot 15, 16\cdot 15^2, 16\cdot 15^3, \ldots, 16\cdot 15^8)^\top$ for $d=10$ and $\z:=(1,5,5\cdot 4, 5\cdot 4^2, 5\cdot 4^3, \ldots, 5\cdot 4^{18})^\top$ for $d=20$. Since the expansion $N$ is a factor in the number of samples of Algorithm~\ref{alg:2dDFT}, cf.\ Corollary~\ref{cor:2dDFTRecoverySublinearSFT}, and we want to concentrate on the dependence on the spatial dimension~$d$, we now fix this parameter to $N:=16$ independent of $d$.
Moreover, the $\texttt{randomScale}$ parameter is set to 0.3. The plots indicate that the numbers of samples grow approximately linearly with respect to the dimension $d$ as stated by Corollaries~\ref{cor:PhaseEncRecoverySublinearSFT} and~\ref{cor:2dDFTRecoverySublinearSFT} for Algorithms~\ref{alg:PhaseEnc} and~\ref{alg:2dDFT}, respectively.
The success rates are slightly better compared to the tests from Figure~\ref{fig:numerics:trigpoly_rand_hc}.

\FloatBarrier
\subsubsection{Random frequency sets within 10-dimensional hyperbolic cross and noisy samples}
\label{sec:numerics:trigpoly_rand_hc_noise}

In this section, we again consider random multivariate trigonometric polynomials with frequencies supported within the hyperbolic cross index set $\mathcal{H}_{33}^{10}$ of expansion $N=33$ and use the reconstructing rank-1 lattice with generating vector~$\z$ as stated in~\eqref{equ:numerics_lattice_z_for_H_33_10} and size $M:=2\,040\,484\,044$.
Similarly as in \cite[Section~5.2]{Kaemmerer2017}, we perturb the samples of the trigonometric polynomial by additive complex (white) Gaussian noise $\varepsilon_j\in\C$ with zero mean and standard deviation $\sigma$.
The noise is generated by 
$\varepsilon_j := \sigma/\sqrt{2}\left(\varepsilon_{1,j}+\mathbbm{ i }\varepsilon_{2,j}\right)$ where $\varepsilon_{1,j},\varepsilon_{2,j}$ are independent standard normal distributed.
Since the signal-to-noise ratio (SNR) can be approximately computed by
$$
\mathrm{SNR} \approx \frac{\sum_{j=0}^{M-1} \vert f(\x_j)\vert^2 / M}{\sum_{j=0}^{M-1} \vert\varepsilon_j\vert^2 / M}
\approx \frac{\sum_{\k\in\supp( c )} \vert c_\k (f)\vert^2}{\sigma^2},
$$
this leads to the choice $\sigma:=\sqrt{\sum_{\k\in\supp( c )} \vert c_\k (f)\vert^2}/\sqrt{\mathrm{SNR}}$ for a targeted SNR value.
The SNR is often expressed in the logarithmic decibel scale (dB), $\mathrm{SNR_{dB}} = 10\, \log_{10} \mathrm{SNR}$ and $\mathrm{SNR} = 10^{\mathrm{SNR_{dB}}/10}$,
i.e., a linear $\mathrm{SNR}=10^2$ corresponds to a logarithmic $\mathrm{SNR_{dB}}=20$
and $\mathrm{SNR}=10^3$ corresponds to $\mathrm{SNR_{dB}}=30$.
Here, we perform tests with sparsity $s=100$ and signal-to-noise ratios $\mathrm{SNR_{dB}}\in\{0,5,10,15,20,25,30\}$.
Moreover, we only use the bandwidth parameter $\tilde{M}=1 + 2 \| \z \|_\infty \max_{ \k \in \mathcal{I} } \| \k \|_1\approx 1.6\cdot 10^{10}$.
Besides that, we choose the algorithm parameters as in Figure~\ref{fig:numerics:trigpoly_rand_hc}.

\begin{figure}[!h]
\begin{subfigure}[c]{0.495\textwidth}
		\begin{tikzpicture}[baseline]
		\begin{axis}[
		font=\footnotesize,
		enlarge x limits=true,
		enlarge y limits=true,
		height=0.8\textwidth,
		grid=major,
		width=\textwidth,
		xtick={0,5,10,15,20,25,30},
		ymin=0,ymax=1,
		xlabel={$\mathrm{SNR}_\mathrm{db}$},
		ylabel={success rate},
		legend style={legend cell align=left}, legend pos=south east,
		legend columns = 1,
		]
		\addplot[blue,mark=square,mark size=2.5pt,mark options={solid}] coordinates {
          (0,0.01) (5,0.14) (10,0.47) (15,0.79) (20,0.90) (25,0.94) (30,0.95)
		};
	    \addlegendentry{phase rs=0.3}
		\addplot[blue,dashed,mark=square,mark size=2.5pt,mark options={solid}] coordinates {
          (0,0.03) (5,0.21) (10,0.61) (15,0.87) (20,0.94) (25,0.96) (30,0.98)
		};
	    \addlegendentry{phase rs=0.5}
		\addplot[red,mark=triangle,mark size=2.5pt,mark options={solid}] coordinates {
          (0,0.04) (5,0.36) (10,0.65) (15,0.82) (20,0.85) (25,0.90) (30,0.90)
		};
	    \addlegendentry{2dim rs=0.3}
		\addplot[red,dashed,mark=triangle,mark size=2.5pt,mark options={solid}] coordinates {
          (0,0.11) (5,0.43) (10,0.79) (15,0.93) (20,0.96) (25,0.98) (30,1.00)
		};
	    \addlegendentry{2dim rs=0.5}
		\end{axis}
		\end{tikzpicture}
\subcaption{success rates vs.\ noise level for $s=100$}
\label{fig:numerics:trigpoly_rand_hc:success_noise}
\end{subfigure}
\hfill
\begin{subfigure}[c]{0.495\textwidth}
		\begin{tikzpicture}[baseline]
		\begin{axis}[
		font=\footnotesize,
		enlarge x limits=true,
		enlarge y limits=true,
		height=0.8\textwidth,
		grid=major,
		width=\textwidth,
		xtick={0,5,10,15,20,25,30},
		ymode=log,
		xlabel={$\mathrm{SNR}_\mathrm{db}$},
		ylabel={relative $\ell^2$ error},
		legend style={legend cell align=left}, legend pos=north east,
		legend columns = 1,
		]
		\addplot[blue,mark=square,mark size=2.5pt,mark options={solid}] coordinates {
          (0,9.933e-02) (5,3.393e-02) (10,1.178e-02) (15,4.968e-03) (20,3.468e-03) (25,3.054e-03) (30,2.384e-03)
		};
	    \addlegendentry{phase rs=0.3}
		\addplot[blue,dashed,mark=square,mark size=2.5pt,mark options={solid}] coordinates {
          (0,5.263e-02) (5,1.522e-02) (10,4.414e-03) (15,1.451e-03) (20,7.057e-04) (25,3.943e-04) (30,2.124e-04)
		};
	    \addlegendentry{phase rs=0.5}
		\addplot[red,mark=triangle,mark size=2.5pt,mark options={solid}] coordinates {
          (0,4.929e-02) (5,1.614e-02) (10,8.974e-03) (15,7.105e-03) (20,6.779e-03) (25,6.385e-03) (30,5.738e-03)
		};
	    \addlegendentry{2dim rs=0.3}
		\addplot[red,dashed,mark=triangle,mark size=2.5pt,mark options={solid}] coordinates {
          (0,2.781e-02) (5,8.678e-03) (10,2.728e-03) (15,1.213e-03) (20,6.579e-04) (25,3.642e-04) (30,1.988e-04)
		};
	    \addlegendentry{2dim rs=0.5}
		\end{axis}
		\end{tikzpicture}
\subcaption{relative $\ell^2$ errors vs.\ noise level for $s=100$}
\label{fig:numerics:trigpoly_rand_hc:error_noise}
\end{subfigure}
\caption{Average success rates (all frequencies detected) and relative $\ell^2$ errors over 100 test runs for Algorithm~\ref{alg:PhaseEnc} with $\mathcal{A}_{ 2s, \tilde M }^\mathrm{sub, MC}$, denoted by ``phase'', and  Algorithm~\ref{alg:2dDFT} with $\mathcal{A}_{ 2s, \tilde M }^\mathrm{sub, MC}$, denoted by ``2dim'', on random multivariate trigonometric polynomials within hyperbolic cross $\mathcal{I}:=\mathcal{H}_{33}^{10}$, setting $\texttt{randomScale}:=\text{rs}\in\{0.3,0.5\}$ and bandwidth parameter $\tilde{M}=1 + 2 \| \z \|_\infty \max_{ \k \in \mathcal{I} } \| \k \|_1\approx 1.6\cdot 10^{10}$.}
\label{fig:numerics:trigpoly_rand_hc:noise}
\end{figure}
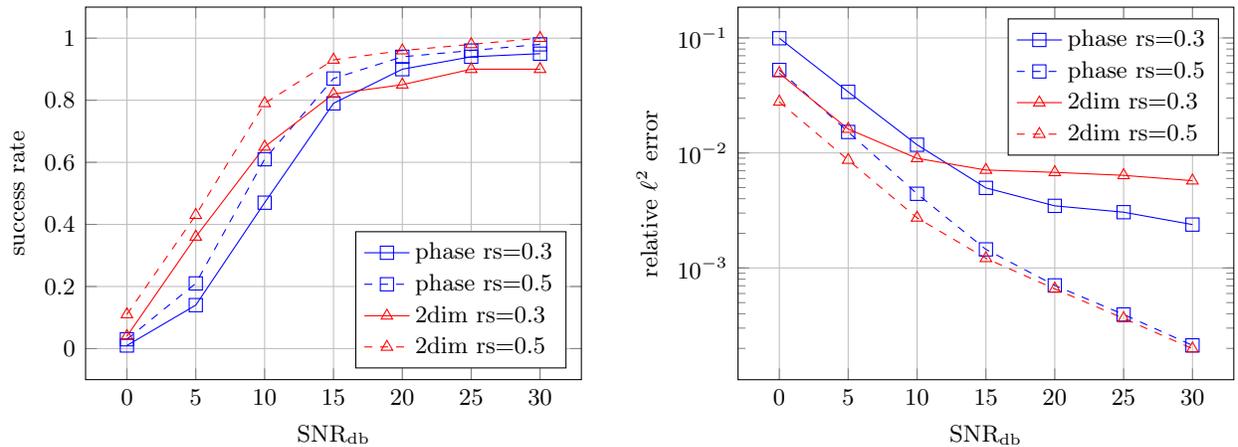

In Figure~\ref{fig:numerics:trigpoly_rand_hc:success_noise}, we visualize the success rates in dependence on the noise level.
For $\texttt{randomScale}\in\{0.3,0.5\}$ and both algorithms, the success rates start at less than 0.12 for $\mathrm{SNR_{dB}}=0$ and grow for increasing signal-to-noise ratios until at least 0.90 for $\mathrm{SNR_{dB}}=30$.
The success rates of Algorithm~\ref{alg:2dDFT} with $\mathcal{A}_{ 2s, \tilde M }^\mathrm{sub, MC}$ (``2dim'') are often higher than for Algorithm~\ref{alg:PhaseEnc} with $\mathcal{A}_{ 2s, \tilde M }^\mathrm{sub, MC}$ (``phase''), which may be caused by the larger numbers of samples for Algorithm~\ref{alg:2dDFT} and the noise model used. Note that the numbers of samples correspond to those in Figure~\ref{fig:numerics:trigpoly_rand_hc} for $s=100$ independent of the noise level.
For Algorithm~\ref{alg:2dDFT} with $\texttt{randomScale}=0.3$, the increase of the success rate seems to stagnate at $\mathrm{SNR_{dB}}=20$, while this does not seem to be the case for $\texttt{randomScale}=0.5$ or Algorithm~\ref{alg:PhaseEnc}.
In particular, this behavior can also be observed in Figure~\ref{fig:numerics:trigpoly_rand_hc:error_noise}, where we plot the average relative $\ell^2$ error of the Fourier coefficients against the signal-to-noise ratio.
Here, we observe that for $\texttt{randomScale}=0.3$, the decrease of the errors for increasing $\mathrm{SNR_{dB}}$ values almost stops once reaching $\mathrm{SNR_{dB}}=20$ for both algorithms. Initially, the average error of Algorithm~\ref{alg:2dDFT} is smaller, but at $\mathrm{SNR_{dB}}=15$ and higher, the average error of Algorithm~\ref{alg:PhaseEnc} is smaller.
In case of $\texttt{randomScale}=0.5$, we observe a distinct decrease for growing signal-to-noise ratios for both algorithms.

\FloatBarrier
\subsubsection{Deterministic frequency set within 10-dimensional hyperbolic cross and noisy samples}

Next, instead of randomly chosen frequencies, we consider frequencies on a $d$-dimensional weighted hyperbolic cross
$$\mathcal{H}_N^{d,\alpha} := \left\{\k \in \mathbbm{Z}^d \colon \prod_{\ell=1}^d \max(1, \ell^\alpha\,|k_\ell|) \leq \frac{N}{2} \quad \text{and} \quad \max_{\ell=1}^d k_\ell<\frac{N}{2}
\right\}.$$
Here, we use $d=10$, $N=33$, $\mathcal{I}:=\mathcal{H}_{33}^{10}$, and $\alpha=1.7$, which yields $s=|\mathcal{H}_{33}^{10,1.7}|=101$. As before, the Fourier coefficients $c_\k$ are randomly chosen from $[-1,1] + \mathbbm{ i }\,[-1,1]$, $|c_\k|\geq 10^{-3}$.
We use the same lattice and bandwidth parameter as in the last subsection as well as the same noise model and parameters.

\begin{figure}[!h]
\begin{subfigure}[c]{0.495\textwidth}
		\begin{tikzpicture}[baseline]
		\begin{axis}[
		font=\footnotesize,
		enlarge x limits=true,
		enlarge y limits=true,
		height=0.8\textwidth,
		grid=major,
		width=\textwidth,
		xtick={0,5,10,15,20,25,30},
		ymin=0,ymax=1,
		xlabel={$\mathrm{SNR}_\mathrm{db}$},
		ylabel={success rate},
		legend style={legend cell align=left}, legend pos=south east,
		legend columns = 1,
		]
		\addplot[blue,mark=square,mark size=2.5pt,mark options={solid}] coordinates {
          (0,0.00) (5,0.12) (10,0.53) (15,0.77) (20,0.95) (25,0.97) (30,0.99)
		};
	    \addlegendentry{phase rs=0.3}
		\addplot[blue,dashed,mark=square,mark size=2.5pt,mark options={solid}] coordinates {
          (0,0.02) (5,0.25) (10,0.69) (15,0.88) (20,0.97) (25,0.99) (30,0.99)
		};
	    \addlegendentry{phase rs=0.5}
		\addplot[red,mark=triangle,mark size=2.5pt,mark options={solid}] coordinates {
          (0,0.06) (5,0.38) (10,0.76) (15,0.91) (20,0.96) (25,0.99) (30,0.99)
		};
	    \addlegendentry{2dim rs=0.3}
		\addplot[red,dashed,mark=triangle,mark size=2.5pt,mark options={solid}] coordinates {
          (0,0.09) (5,0.49) (10,0.85) (15,0.93) (20,0.98) (25,0.99) (30,1.00)
		};
	    \addlegendentry{2dim rs=0.5}
		\end{axis}
		\end{tikzpicture}
\subcaption{success rates vs.\ noise level for $s=100$}
\label{fig:numerics:trigpoly_det_whc:success_noise}
\end{subfigure}
\hfill
\begin{subfigure}[c]{0.495\textwidth}
		\begin{tikzpicture}[baseline]
		\begin{axis}[
		font=\footnotesize,
		enlarge x limits=true,
		enlarge y limits=true,
		height=0.8\textwidth,
		grid=major,
		width=\textwidth,
		xtick={0,5,10,15,20,25,30},
		ymode=log,
		xlabel={$\mathrm{SNR}_\mathrm{db}$},
		ylabel={relative $\ell^2$ error},
		legend style={legend cell align=left}, legend pos=north east,
		legend columns = 1,
		]
		\addplot[blue,mark=square,mark size=2.5pt,mark options={solid}] coordinates {
          (0,8.628e-02) (5,2.618e-02) (10,7.282e-03) (15,2.484e-03) (20,8.688e-04) (25,4.755e-04) (30,2.540e-04)
		};
	    \addlegendentry{phase rs=0.3}
		\addplot[blue,dashed,mark=square,mark size=2.5pt,mark options={solid}] coordinates {
          (0,4.885e-02) (5,1.438e-02) (10,3.904e-03) (15,1.406e-03) (20,6.471e-04) (25,3.517e-04) (30,2.005e-04)
		};
	    \addlegendentry{phase rs=0.5}
		\addplot[red,mark=triangle,mark size=2.5pt,mark options={solid}] coordinates {
          (0,3.950e-02) (5,1.213e-02) (10,3.717e-03) (15,1.656e-03) (20,8.436e-04) (25,4.471e-04) (30,2.540e-04)
		};
	    \addlegendentry{2dim rs=0.3}
		\addplot[red,dashed,mark=triangle,mark size=2.5pt,mark options={solid}] coordinates {
          (0,2.509e-02) (5,7.828e-03) (10,2.359e-03) (15,1.229e-03) (20,6.299e-04) (25,3.517e-04) (30,1.956e-04)
		};
	    \addlegendentry{2dim rs=0.5}
		\end{axis}
		\end{tikzpicture}
\subcaption{relative $\ell^2$ errors vs.\ noise level for $s=100$}
\label{fig:numerics:trigpoly_det_whc:error_noise}
\end{subfigure}
\caption{Average success rates (all frequencies detected) and relative $ \ell^2 $ errors over 100 test runs for Algorithm~\ref{alg:PhaseEnc} with $\mathcal{A}_{ 2s, \tilde M }^\mathrm{sub, MC}$, denoted by ``phase'', and  Algorithm~\ref{alg:2dDFT} with $\mathcal{A}_{ 2s, \tilde M }^\mathrm{sub, MC}$, denoted by ``2dim'', on multivariate trigonometric polynomials with (deterministic) frequencies on weighted hyperbolic cross within hyperbolic cross $\mathcal{I}:=\mathcal{H}_{33}^{10}$, setting $\texttt{randomScale}:=\text{rs}\in\{0.3,0.5\}$ and bandwidth parameter $\tilde{M}=1 + 2 \| \z \|_\infty \max_{ \k \in \mathcal{I} } \| \k \|_1\approx 1.6\cdot 10^{10}$.}
\label{fig:numerics:trigpoly_det_whc:noise}
\end{figure}
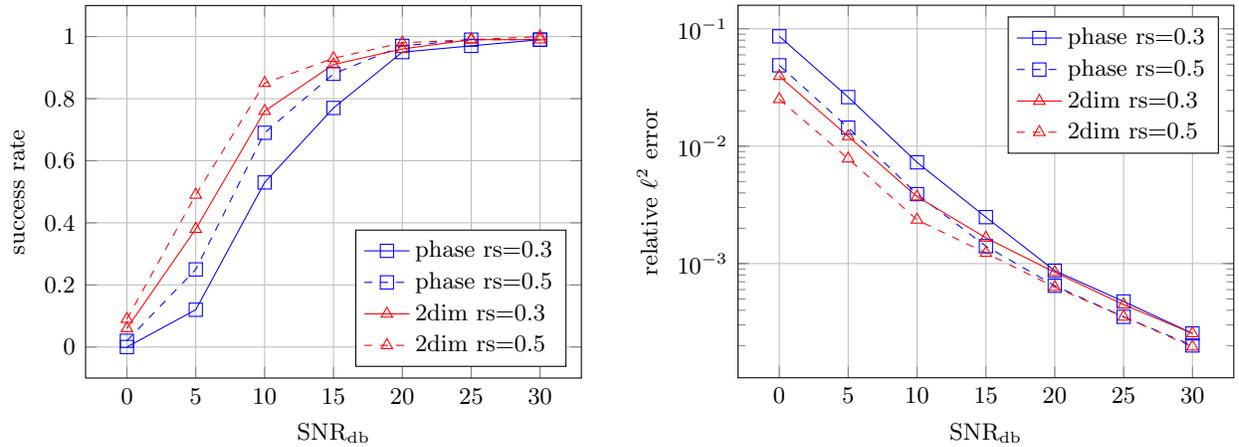

In Figure~\ref{fig:numerics:trigpoly_det_whc:noise}, we depict the obtained results. In particular, the results in Figure~\ref{fig:numerics:trigpoly_det_whc:success_noise} are very similar to the ones for randomly chosen frequencies in Figure~\ref{fig:numerics:trigpoly_rand_hc:success_noise}.
For the case of deterministic frequencies in Figure~\ref{fig:numerics:trigpoly_det_whc:success_noise}, the success rates are slightly better. Moreover, we do not observe the ``stagnation'' of the success rates for Algorithm~\ref{alg:2dDFT} with $\texttt{randomScale}=0.3$. Correspondingly, the relative $\ell^2$ errors, as shown in Figure~\ref{fig:numerics:trigpoly_det_whc:error_noise}, decrease distinctly for growing signal-to-noise ratios.
Algorithm~\ref{alg:2dDFT} performs slightly better than Algorithm~\ref{alg:PhaseEnc}, but also requires more than one order of magnitude more samples, similar to the results shown in Figure~\ref{fig:numerics:trigpoly_rand_hc} for $s=100$.

\subsection{Compressible case in 10 dimensions}

In this section, we apply the methods on a test function which is not exactly sparse but compressible. In addition, we also consider noisy samples as in Section~\ref{sec:numerics:trigpoly_rand_hc_noise}.
We use the 10-variate periodic test function $f\colon\mathbbm{T}^{10}\rightarrow\mathbbm{R}$,
\begin{equation} \label{equ:f:10}
	f(\x):=\prod_{\ell\in\{0,2,7\}}N_2(x_\ell) + \prod_{\ell\in\{1,4,5,9\}}N_4(x_\ell) + \prod_{\ell\in\{3,6,8\}}N_6(x_\ell),
\end{equation}
from \cite[Section~3.3]{Potts2016} and \cite[Section~5.3]{Kaemmerer2017} with infinitely many non-zero Fourier coefficients~$c_\k(f)$,
where $N_m:\mathbbm{T}\rightarrow\mathbbm{R}$ is the B-Spline of order $m\in\mathbbm{N}$,
$$N_m(x) := C_m \sum_{k\in\mathbbm{Z}} \operatorname{sinc}\left(\frac{\pi}{m}k\right)^m (-1)^k \,\mathrm{e}^{2\pi\mathrm{i}kx},$$
with a constant $C_m>0$ such that $\Vert N_m \Vert_{L^2(\mathbbm{T})}=1$.
We remark that each B-Spline $N_m$ of order $m\in\mathbbm{N}$ is a piece-wise polynomial of degree $m-1$.
We apply Algorithm~\ref{alg:PhaseEnc} with $\mathcal{A}_{ 2s, \tilde M }^\mathrm{sub, MC}$ and use the sparsity parameters $s\in\{50,100,250,500,1000,2000\}$, which corresponds to $2s\in\{100,200,500,1000,2000,4000\}$ frequencies and Fourier coefficients for the output of Algorithm~\ref{alg:PhaseEnc}.
We use the frequency set $\mathcal{I}:=\mathcal{H}_{33}^{10}$ and $\texttt{randomScale}:=\text{rs}\in\{0.05,0.1\}$.
Moreover, we work with the same rank-1 lattice as in Section~\ref{sec:numerics:trigpoly_rand_hc_noise}.

The obtained basis index sets $\supp(\mathbf{b})$ should ``consist of'' the union of three lower dimensional manifolds,
a three-dimensional hyperbolic cross in the dimensions $1,3,8$;
a four-dimensional hyperbolic cross in the dimensions $2,5,6,10$;
and a three-dimensional hyperbolic cross in the dimensions $4,7,9$.
All tests are performed 100 times and the relative $L^2$ approximation error
$$
\frac{ \Vert f- \sum_{\k\in\supp(\mathbf{b})} b_\k \, \mathbbm{ e }^{ 2 \pi \mathbbm{ i } \k \cdot \circ } \Vert_{L^2} }{ \Vert f\Vert_{L^2} }
=
\frac{
\sqrt{\Vert f\Vert_{L^2}^2 - \sum_{\k\in \supp(\mathbf{b})}\vert c_\k (f)\vert^2 + \sum_{\k\in \supp(\mathbf{b})}\vert b_{\k}-c_{\k} (f)\vert^2} }{ \Vert f\Vert_{L^2} }
$$
is computed each time.

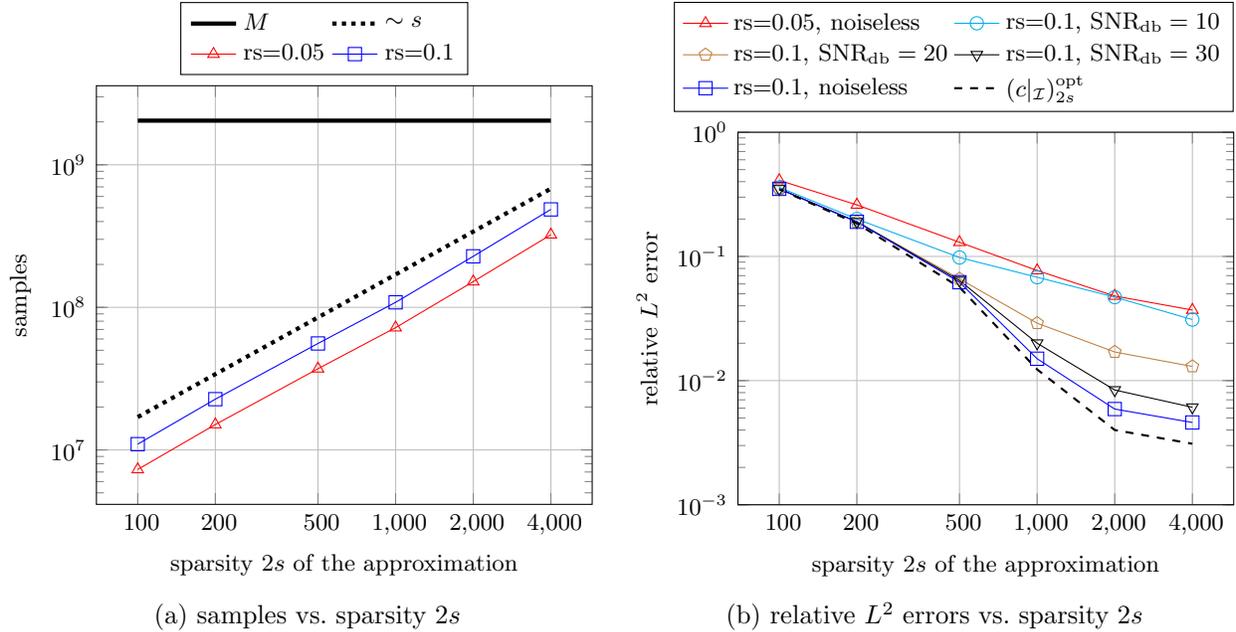
\begin{figure}[!h]
\begin{subfigure}[c]{0.495\textwidth}
		\begin{tikzpicture}[baseline]
		\begin{axis}[
		font=\footnotesize,
		enlarge x limits=true,
		enlarge y limits=true,
		height=0.875\textwidth,
		grid=major,
		width=\textwidth,
		xtick={100,200,500,1000,2000,4000},
		xmode=log,
		ymode=log,
		xticklabel={
			\pgfkeys{/pgf/fpu=true}
			\pgfmathparse{exp(\tick)}%
			\pgfmathprintnumber[fixed relative, precision=3]{\pgfmathresult}
			\pgfkeys{/pgf/fpu=false}
		},
		xlabel={sparsity $2s$ of the approximation},
		ylabel={samples},
		legend style={legend cell align=left, at={(0.75,1.2)}},
		legend columns = 2,
		]
		\addplot[black,no marks,ultra thick] coordinates {
          (100,2040484044) (4000,2040484044)
        };
        \addlegendentry{$M$}
        \addplot[dotted,ultra thick,samples=100,black,domain=100:4000] {1.7e5*x};
        \addlegendentry{$\sim s$}
		\addplot[red,mark=triangle,mark size=2.5pt,mark options={solid}] coordinates {
         (100,7294709.070) (200,15071065.240) (500,37117427.600) (1000,72078792.390) (2000,151559318.900) (4000,323125263.560)
        };
        \addlegendentry{rs=0.05}
		\addplot[blue,mark=square,mark size=2.5pt,mark options={solid}] coordinates {
          (100,10979401.950) (200,22700795.370) (500,55816226.620) (1000,108439173.810) (2000,228025892.160) (4000,486422588.960)
		};
	    \addlegendentry{rs=0.1}
		\end{axis}
		\end{tikzpicture}
\subcaption{samples vs.\ sparsity $2s$}
	\label{fig:approx_sparse:fourier_d10:samples_vs_s}
\end{subfigure}
	\hfill
\begin{subfigure}[c]{0.495\textwidth}
		\begin{tikzpicture}[baseline]
		\begin{axis}[
		font=\footnotesize,
		enlarge x limits=true,
		height=0.8\textwidth,
		grid=major,
		width=\textwidth,
		xtick={100,200,500,1000,2000,4000},
		ymin=1e-3,ymax=1,
		xmode=log,
		ymode=log,
		xticklabel={
	\pgfkeys{/pgf/fpu=true}
	\pgfmathparse{exp(\tick)}%
	\pgfmathprintnumber[fixed relative, precision=3]{\pgfmathresult}
	\pgfkeys{/pgf/fpu=false}
},
		xlabel={sparsity $2s$ of the approximation},
		ylabel={relative $L^2$ error},
		legend style={legend cell align=left, at={(1,1.35)}},
		legend columns = 2,
		]
		\addplot[red,mark=triangle,mark size=2.5pt,mark options={solid}] coordinates {
          (100,4.1e-01) (200,2.6e-01) (500,1.3e-01) (1000,7.7e-02) (2000,4.8e-02) (4000,3.7e-02)
		};
	    \addlegendentry{rs=0.05, noiseless}
        \addplot[cyan,mark=o,mark size=2.5pt,mark options={solid}] coordinates {
          (100,3.6e-01) (200,2.0e-01) (500,9.8e-02) (1000,6.8e-02) (2000,4.7e-02) (4000,3.1e-02)
        };
        \addlegendentry{rs=0.1, $\mathrm{SNR}_\mathrm{db}=10$}
        \addplot[brown,mark=pentagon,mark size=2.5pt,mark options={solid}] coordinates {
          (100,3.5e-01) (200,1.9e-01) (500,6.6e-02) (1000,2.9e-02) (2000,1.7e-02) (4000,1.3e-02)
        };
        \addlegendentry{rs=0.1, $\mathrm{SNR}_\mathrm{db}=20$}
        \addplot[black,mark=triangle,mark size=2.5pt,mark options={solid,rotate=180}] coordinates {
          (100,3.5e-01) (200,1.9e-01) (500,6.4e-02) (1000,2.0e-02) (2000,8.4e-03) (4000,6.1e-03)
        };
        \addlegendentry{rs=0.1, $\mathrm{SNR}_\mathrm{db}=30$}
		\addplot[blue,mark=square,mark size=2.5pt,mark options={solid}] coordinates {
          (100,3.5e-01) (200,1.9e-01) (500,6.2e-02) (1000,1.5e-02) (2000,5.9e-03) (4000,4.6e-03)
		};
	    \addlegendentry{rs=0.1, noiseless}
		\addplot[black, no marks, dashed, thick] coordinates {
          (100,0.3471) (200,0.1843) (500,0.0565) (1000,0.0123) (2000,0.0040) (4000,0.0031)
		};
	    \addlegendentry{$(c\rvert_{ \mathcal{I} })_{2s}^\mathrm{opt}$}
		\end{axis}
		\end{tikzpicture}
\subcaption{relative $L^2$ errors vs.\ sparsity $2s$}
	\label{fig:approx_sparse:fourier_d10:error_vs_s}
\end{subfigure}
	\caption{%
Average number of samples and relative $L^2$ errors over 100 test runs for Algorithm~\ref{alg:PhaseEnc} with $\mathcal{A}_{ 2s, \tilde M }^\mathrm{sub, MC}$ on 10-dimensional test function~\eqref{equ:f:10} consisting of tensor products of B-Splines of different order. Search space is unweighted hyperbolic cross $\mathcal{I}:=\mathcal{H}_{33}^{10}$ with SFT parameters $\texttt{randomScale}:=\text{rs}\in\{0.3,0.5\}$ and $ \tilde{M}=1 + 2 \| \z \|_\infty \max_{ \k \in \mathcal{I} } \| \k \|_1\approx 1.6\cdot 10^{10}$. }
	\label{fig:approx_sparse:fourier_d10:samples_error_vs_s}
\end{figure}

In Figure~\ref{fig:approx_sparse:fourier_d10:samples_vs_s}, we visualize the average number of samples against the sparsity $2s$ of the approximation. We observe an almost linear increase with respect to $2s$.
In Figure~\ref{fig:approx_sparse:fourier_d10:error_vs_s}, we show the average relative errors for $\texttt{randomScale}\in\{0.05,0.1\}$ in the noiseless case as well as $\texttt{randomScale}=0.1$ for $\mathrm{SNR}_\mathrm{db}\in\{10,20,30\}$. In general, for increasing sparsity, the errors become smaller. For $\texttt{randomScale}=0.05$ in the noiseless case and $\texttt{randomScale}=0.1$ with $\mathrm{SNR}_\mathrm{db}=10$, the average error are similar and stay above $3\cdot 10^{-2}$ even for sparsity $2s=4000$. For higher signal-to-noise ratio, the error decreases further. For $\mathrm{SNR}_\mathrm{db}=30$, the obtained average error is $6.1\cdot 10^{-3}$ for $2s=4000$, which is only approximately twice as high as the best possible error when using the $2s$ largest (by magnitude) Fourier coefficients $c_\k(f)$ with the restriction $\k\in\mathcal{I}:=\mathcal{H}_{33}^{10}$. The latter is plotted in Figure~\ref{fig:approx_sparse:fourier_d10:error_vs_s} as dashed line without markers.

\bibliographystyle{abbrv}
\bibliography{References}

\end{document}